\documentclass[A4paper, 12pt]{amsart}
\usepackage[latin1]{inputenc}   
\usepackage{amssymb,bbm,amsmath,amsthm}
\usepackage{latexsym}

\usepackage[all]{xy}

\usepackage{graphicx}

\usepackage{epic} 
\usepackage{eepic}
\usepackage{xcolor}
\usepackage{hyperref}
\usepackage{verbatim}

\input xy
\xyoption{all}

\newcommand{\R}{\ensuremath{\mathbb{R}}}

\newcommand{\C}{\ensuremath{\mathbb{C}}}
\newcommand{\Q}{\ensuremath{\mathbb{Q}}}
\newcommand{\N}{\ensuremath{\mathbb{N}}}
\newcommand{\F}{\ensuremath{\mathbb{F}}}

\newcommand{\0}{\ensuremath{\overline{0}}}
\newcommand{\fa}{\mathfrak{a}}
\newcommand{\fb}{\mathfrak{b}}

\newcommand{\fp}{\mathfrak{p}}
\DeclareMathOperator{\Spec}{Spec}

\DeclareMathOperator{\Hom}{Hom}

\newtheorem{Thm}{Theorem}[section]

\newtheorem{Cor}[Thm]{Corollary}

\newtheorem{Lem}[Thm]{Lemma}

\newtheorem{Prop}[Thm]{Proposition}

\newtheorem{Conj}[Thm]{Conjecture}

{\theoremstyle{definition}
       \newtheorem{defi}[Thm]{Definition}

       \newtheorem{Rmk}[Thm]{Remark}
       \newtheorem{ex}[Thm]{Example}
    
        \newtheorem{lemma}[Thm]{Lemma}
    
    \newtheorem{Question}[Thm]{Question}

\numberwithin{equation}{Thm}
\bibliographystyle{acm}
\title{Frobenius  Splitting in Commutative Algebra}
\author{Karen E. Smith and Wenliang Zhang}
\begin{document}
\maketitle
\tableofcontents

The purpose of these lectures is to give a gentle introduction to Frobenius splitting, or more  broadly ``Frobenius techniques,"  for beginners.  Frobenius  splitting has inspired a vast arsenal of techniques in commutative algebra, algebraic geometry, and representation theory. 
   Many related techniques have been developed by different camps of researchers, often using different language and notation. Although 
   there are great number of technical papers and books written over the past forty years,  many of the most elegant ideas, and the connections between them,  have  coalesced only in the past decade. We wish to bring this emerging simplicity  to the 	uninitiated. 
   
Our story of  Frobenius splitting begins in the 1970s, with  the proof of the celebrated Hochster-Roberts' theorem on the Cohen-Macaulayness of  rings of invariants \cite{HochsterRobertsInvSubring}. This proof, in turn, was inspired by Peskine and Spziro's ingenious  use of the iterated Frobenius---or $p$-th power--- map to prove a constellation of ``intersection conjectures" due to Serre \cite{ PeskineSzpiroHomConjectures}.  Mehta and Ramanathan coined the term ``Frobenius splitting"  a decade later in a beautiful paper which moved beyond the affine case  to prove theorems about Schubert varieties and other important topics in the representation theory of algebraic groups \cite{Mehta-Ramanathan}. 
Although these ``characteristic $p$ techniques" are powerful also for proving theorems for algebras and varieties over fields of characteristic zero,  our lectures focus on the prime characteristic case, since the technique of  reduction to characteristic $p$ has now become fairly standard.

The first lecture treats Frobenius splitting  in the local algebraic spirit  of Hochster and Roberts, where it provides a tool for controlling the singularities of a local ring.  We prove the Hochster-Roberts' theorem, giving what is essentially Hochster and Huneke's tight closure proof without explicitly mentioning tight closure \cite{HochsterHunekeJAMS}.  Roughly the point is that  a ring of invariants (of a linearly reductive group) is a direct summand of a polynomial ring,  and as such inherits a strong form of Frobenius splitting called {\it F-regularity}, which in turn, implies Cohen-Macaulayness.

The second lecture considers Frobenius splitting  for schemes in the  spirit of Mehta and Ramanathan, emphasizing how the Frobenius map can be used to study {\it global } properties of a smooth projective scheme. For example, we show how  Frobenius splitting leads to powerful vanishing theorems for cohomology of line bundles, and prove structure theorems for Frobenius split and globally F-regular  projective varieties.  We explain  how these local and global  Frobenius tools  are equivalent, despite developing  independently during the last decades of the twentieth century.  (This separate development is evidenced by the surprising disjointness of the  two  monographs ``Tight Closure and its Applications" by Huneke in 1996 \cite{HunekeTightClosure}, and ``Frobenius Splitting Methods in Geometry and Representation Theory" by Brion and Kumar, 2005 \cite{BrionKumar}).   The two schools independently discovered the same ideas, though often in language almost impenetrable to the other: Ramanathan's notion of {\it Frobenius splitting along a divisor} \cite{RamanathanFSplittingandSchubertVar} is closely related to Hochster and Huneke's {\it strong F-regularity} \cite{HochsterHunekeTCStrongFRegularity};    compatibly split subschemes turn out to be essentially dual to modules with Frobenius action in commutative algebra; and criteria  for Frobenius splitting  of  projective varieties   (in terms of  pluri-canonical sections)   amount to dual  criteria in terms of Frobenius actions on the injective hull of the  residue field in a local ring.   Both points of view are enhanced by understanding the connections between them. 
       
     The third and fourth lecture treat {\it test ideals,\/}  special ideals (subschemes) carrying information about the action of Frobenius.  Although test ideals first arose  on the commutative algebra side as a technical component of local tight closure theory, more recent work of Karl Schwede has led to deeper understanding of the test ideal in the broader context of Frobenius splitting for schemes \cite{SchwedeCentersOfFPurity}. Test ideals can be viewed as a prime characteristic analog of multiplier ideals  (\cite{SmithMulplierIdealUniversalTestIdeal}, \cite{HaraGeometricInterpTestideals})  and also of log-canonical centers for complex varieties  \cite{SchwedeCentersOfFPurity}, depending on the context.   Indeed in recent years, this connection  is of increasing interest in the minimal model program in characteristic $p$. The third lecture develops the ``absolute test ideal" as one ideal in a distinguished lattice of ideals well-behaved with respect to the Frobenius map. The fourth lecture develops test ideals for {\it pairs} in the important special case where the ambient ring is regular. We present simple  proofs of all the basic properties
     in the setting of an {\it regular} ambient regular ring (much more technical proofs   are  scattered throughout the literature as special  cases of more general results due to Hara, Watanabe, Takagi and Yoshida).  We include a self-contained development of an asymptotic theory of test ideals analogous to the story of asymptotic multiplier ideals developed in \cite{ELS1}, including an application to  the behavior of symbolic powers of ideals in a regular ring.    
     
        This article is not in any way intended to compete with the excellent surveys
    \cite{SchwedeTuckerSurvey} or \cite{BlickleSchwedeSurvey}, both of which contain a more technical and extensive survey of recent ideas. There are also older surveys such as \cite{SmithSantaCruz} which explain more about reduction to characteristic $p$ in this context and the connections between singularities in the minimal model program and characteristic $p$ techniques. Other possible surveys of interest include Huneke's lectures on Tight Closure  \cite{HunekeTightClosure}, Brion-Kumar's text on Mehta-Ramanathan's Frobenius splitting  \cite{BrionKumar}, Swanson's notes on Tight Closure \cite{SwansonTightclosure}, Huneke's survey on F-signature and Hilbert-Kunz multiplicities \cite{HunekeLectureHKFSignature},   the surveys
    \cite{BenitoFaberSmith} (more basic) or \cite{MustataIMPANGA} (more advanced) on log canonical and F-threshold, or Holger Brenner's survey on geometric methods in tight closure theory \cite{BrennerTCVB}. 
    
    We are grateful to Angelica Benito, Daniel Hern\'andez, Greg Muller, Luis N\'u\~nez, Karl Schwede,  Felipe P\`erez, and Emily Witt  for carefully reading a draft of this paper and making suggestions to improve it. 
\section{The Frobenius Map for Rings: the local theory.}

Let $R$ be any commutative ring of prime characteristic $p$.  The Frobenius map is the $p$-th power map:
$$
F: R \rightarrow R; \,\,\,\,\,\,\,\,\,\,\,\,\,\,\,\,
r \mapsto r^p.
$$
Because $(r + s)^p = r^p + s^p$ in characteristic $p$, the Frobenius map is a {\it ring homomorphism.}  Its image is the subring $R^p$  of all elements of $R$ that are $p$-th powers. We thus have an inclusion of rings $R^p \hookrightarrow R$. 

Our goal is to use the Frobenius map---or more precisely the {\bf $R^p$-module structure of $R$}---to understand the singularities of the ring $R$. Typically,   thankfully,  this module is finitely generated:

\begin{defi}\label{F-finite}
A ring $R$  of positive characteristic $p$ is said to be {\bf F-finite}  if $R$ is a finitely generated module over its subring $R^p$. 
\end{defi}

F-finiteness is a mild assumption, usually satisfied in ``geometric" settings. For example, a perfect field $k$ is F-finite of course: by definition $k^p = k$. Every ring finitely generated over an F-finite ring is F-finite. In particular, finitely generated algebras over perfect fields are F-finite.
Moreover, it is easy to see that the class of F-finite rings is closed under localization, surjective image, completion at a maximal ideal, and finite extensions.

Already we can observe that 
one of the most basic singularities of the ring $R$ can be detected by Frobenius. Namely $R$ is {\it reduced} (meaning that $0$ is the only nilpotent)  if and only if the Frobenius map is injective.

A  less trivial observation, and indeed the starting point for  using Frobenius to classify singularities,  is the  famous 1969 theorem of  Ernst Kunz:

\begin{Thm}\label{kunzthm}\cite{Kunz} 
An F-finite ring $R$ is regular if and only if $R$ is locally free as an $R^p$-module.{\footnote{More generally, even if $R$ is not F-finite, Kunz shows that a  ring of characteristic $p >0$ is regular if and only if its Frobenius map is flat.  
}}
\end{Thm}

\begin{ex} Let $R$ be the ring $\mathbb F_p[x, y]$. Considered as a module over the subring $R^p =  \mathbb F_p[x^p, y^p]$, it is easy to see that $R$ is a free $R^p$-module. Indeed, the monomials 
$$\{x^ay^b; \,\,\,\,\,\,\,  0 \leq a, b < p \}$$
form a free basis.  Similarly, a polynomial ring $R$ in $d$ variables over  $\mathbb F_p$ is a free-module over $R^p$ of rank $p^d$. \end{ex}

As we will see, it is possible to classify the singularities of $R$ according to how far the $R^p$-module $R$ is from free. This is one of the ways in which the local theory of ``Frobenius techniques" takes shape. 

\subsubsection{Notation}
For simplicity, let us assume that  $R$ is a domain. In this case, the inclusion of $R^{p}$-modules $R^p \hookrightarrow R$ is entirely equivalent to the inclusion of $R$-modules 
$R \hookrightarrow R^{1/p}$, where $R^{1/p}$ is the subring of $p$-th roots of elements of $R$ in an  algebraic closure of the fraction field of $R$ (note that each $r \in R$ has a {\it unique\/} $p$-th root).
Thus to understand the $R^p$ module structure of $R$ is  to understand the $R$-module structure of $R^{1/p}$. Both are equivalent to viewing $R$ as an $R$-module via restriction of scalars by the Frobenius map $F: R \rightarrow R$. When using this last point of view, it can be useful to notationally distinguish the source and target copies of $R$; my favorite way  (because it is consistent with the standard notation used for maps of schemes as in \cite{Hartshorne}) is to use $F_*R$ for the target copy of $R$, so that the Frobenius map is 
denoted $F: R \rightarrow F_*R$.
It is worth  being open to any of these notations, since depending on the situation, one may be more illuminating  than another.{\footnote{In the tight closure literature, the notation $F_*R$ is often replaced by ${}^1R$, so the Frobenius map is written $R \rightarrow {}^1R$, with the second copy of $R$ denoting  $R$ as an abelian group with $R$-module structure $r \cdot x = r^px$ where $r\in R$ and $x\in {}^1R$. The notation $R^{1/p}$ can also be used to denote the target copy of $R$ in general; if $R$ is reduced, each element has a unique $p$-th root inside the total ring of quotients so that the Frobenius map beomes the inclusion $R \hookrightarrow R^{1/p}$, but this notation can be misleading if $R$ is not reduced because then the Frobenius  map $R\rightarrow R^{1/p}$ is not injective. }}   On the other hand, sometimes it is convenient {\it not} to notationally distinguish the source and target of Frobenius; this point of view is at the heart of Blickle's theory of Cartier modules. Cf. Remark \ref{Cartier}.

\subsection{Splitting.}

Let
$R \rightarrow  S$
be any  homomorphism of rings. Considering $S$ as an $R$-module via restriction of scalars,  we can ask whether or not
this map {\it splits \/} in the category of $R$-modules.

\begin{defi}
We say that $R \rightarrow S$ {\it splits\/}
if there is an $R$-module map $ S \overset{\phi}\rightarrow R$ such that the composition
$$
R \rightarrow S \overset{\phi}\rightarrow R
$$ is the identity map on $R$.
 Equivalently, $R \rightarrow  S$ splits if there exists   $\phi \in \Hom_{R}(S, R)$ such that $\phi(1) = 1$.  
 \end{defi}
 
 If $R \rightarrow S$ splits, then it is obviously injective, so we often restrict attention to inclusions of rings. Given an inclusion $R \hookrightarrow S$, we also say    {\it ``$R$ is a direct summand of $S$"} to mean that the map splits. This concept is important because many nice properties of rings pass to direct summands. 

\begin{ex}
Let $G$ be a finite group acting on a ring $S$. Let $S^G$ denote the ring of invariants, that is, the subring of elements of $S$ that are fixed by the action of $G$. 
The reader will easily show that  the map
$$\varphi:S\to S^G\,\,\,\,\,\,\, \varphi(s)=\frac{1}{|G|}\sum_{g\in G}g\cdot s$$
gives a splitting of $S^G\hookrightarrow S$, provided that  $|G|$ is invertible in $S$. \end{ex}

\begin{defi}
A ring $R$ of characteristic $p$ is  {\bf Frobenius split}  (or {\bf F-split}) if the Frobenius map splits. 
Explicitly, a reduced ring $R$ is Frobenius  split if the  ring inclusion  $R^p \hookrightarrow R$ splits as a map of $R^p$-modules. 
Equivalently,  a reduced ring $R$ is Frobenius split if there exists $\pi \in \Hom_{R}(R^{1/p}, R)$ such that $\pi(1) = 1$. 
\end{defi}

\begin{ex} \label{ex3}The ring $R = \mathbb F_p[x, y]$ is Frobenius split.   Indeed, we have seen that $R$ is free over the subring $R^p =  \mathbb F_p[x^p, y^p]$, with  basis
$\{x^ay^b\}$ where $ 0 \leq a, b < p$. Any projection onto the summand generated by the basis element $1 = x^0y^0$ gives a splitting.
\end{ex}

Frobenius splitting is a local condition on a ring{\footnote{But use caution: on a non-affine scheme the Frobenius  map can split locally at each point but not globally! For example, a smooth projective curve is locally Frobenius split, but not globally Frobenius split if the genus is greater than one; see Example \ref{curvesummary}. }}:
\begin{Lem} \label{loc} Let $R$ be any F-finite ring of prime characteristic. 
The locus of points $P $ in Spec $R$ such that $R_P$ is Frobenius split is an open set. In particular,   $R$ is Frobenius split if and only if for all maximal (respectively, prime) ideals $P$ in Spec $R$, the local ring $R_P$ is Frobenius split. \end{Lem}

The proof of Lemma \ref{loc} is easy, following from the fact that $R^{1/p}$ is a finitely generated $R$ module, so we leave it to the reader.
Using Lemma \ref{loc},  it is not hard to prove the following generalization of Example \ref{ex3}:

\begin{Prop}\label{RegimpliesF-split} Every F-finite  regular ring is Frobenius split.
\end{Prop}

Indeed, for a local ring $(R, m)$, we can think of regularity as the condition that  the $R$-module $R^{1/p}$ decomposes completely into a direct sum of copies of $R$, whereas Frobenius splitting is the condition that $R^{1/p}$ contains {\it at least one direct sum copy\/} of $R$. 

The property of Frobenius splitting is passed on to direct summands: 

\begin{Prop}\label{split} Let $R \subset S$ be any inclusion of rings of characteristic $p$ which splits in the category of $R$-modules. If $S$ is Frobenius split, then $R$ is Frobenius split.
\end{Prop}

\begin{proof} We have a commutative diagram:
$$
\xymatrix{
R   \ar@{^{(}->}[r]   & S   \\
R^p  \ar@{^{(}->}[u] \ar@{^{(}->}[r] &   S^p \ar@{^{(}->}[u]
}
$$
If we denote the splitting of $R \hookrightarrow S$ by $\phi$, then the map $R^p  \hookrightarrow
S^p$ is also split, by the map  $\phi^p$ defined by taking the $p$-th powers of everything. Our assumption that $S$ is Frobenius split amounts to the existence of  an $S^p$-linear  map $\pi: S \rightarrow S^p$ sending $1$ to $1$. The composition $\phi^p \circ \pi$, when restricted to $R$, gives an $R^p$-linear map from $R$ to $R^p$ sending $1$ to $ 1$. Thus $R$ is also Frobenius split.
\end{proof}

With Property \ref{split} in hand, it is easy to construct examples of Frobenius split rings which are not regular: a direct summand of a regular ring is always Frobenius split but not usually regular. For example,

\begin{ex}\label{ex4} For any graded ring $R = \bigoplus_{n\in \mathbb N} R_n$, the inclusion of any Veronese subring 
$$R^{(d)} =  \bigoplus_{n\in \mathbb N} R_{dn} \hookrightarrow R$$ splits. So 
a Veronese subring of a polynomial ring is Frobenius split in any characteristic,  although such a subring is rarely regular. For example, $\frac{k[x, y, z]}{(xz - y^2)} \cong   k[u^2, uv, v^2] \subset k[u, v]$ is the second Veronese subring of a polynomial ring, hence Frobenius split in every characteristic (but never regular). 
\end{ex}

\begin{Rmk} 
Frobenius splitting was first systematically studied by Hochster and Roberts in the 1970s. See  \cite{HochsterRobertsInvSubring}, \cite{HochsterRobertsFPurityLC}. 
The term {\it  Frobenius split} however, was introduced a decade later by Mehta and Ramanathan, in their  beautiful paper \cite{Mehta-Ramanathan} which interpreted many of these ideas in a projective setting. Hochster and Roberts actually introduced a slightly more technical notion called {\it F-purity}, which (as they show) is equivalent to Frobenius splitting  under the F-finite hypothesis. For non-F-finite rings, the notion of  F-purity  is {\it a priori\/} weaker than Frobenius splitting; however,  we do not know a single (excellent) example of an F-pure ring which is not Frobenius split. 
\end{Rmk}

\subsection{Iterations of Frobenius and F-regularity.}
The real power of Frobenius emerges when we iterate it.
The composition of the Frobenius map with itself is obviously a ring homomorphism sending each element $r$ to $(r^p)^p = r^{p^2}$. More generally, 
 for each natural number $e$,  the iteration of Frobenius $e$ times is the ring homomorphism
$$F^e: R \rightarrow R \,\,\,\,\,\,\,\,\,\,\,\, r \mapsto r^{p^e}.$$ The images of each of these iterates 
 produces an infinite descending chain of subrings
$$
R \supset R^p \supset R^{p^2} \supset  R^{p^3}  \supset \dots
$$
The original ring $R$ can be viewed  as a  module over each of these subrings $R^{p^e}.$ Indeed, assuming that $R$ is F-finite, then also  $R$ is  finitely generated as an $R^{p^e}$-module for each $e$. 
Again, understanding the $R^{p^e}$-module structure of the $R^{p^e}$-module $R$ is essentially the same as understanding the $R$-module structure of the $R$-module $R^{1/p^e}$ 
(or $F^e_*R$). 

If $R$ is F-finite, Kunz's theorem implies that  $R$ is regular if and only if  the $R$-modules $R^{1/p^e}$ are all locally free. Classes of  ``F-singularities" can be defined depending on the extent to which the $R^{1/p^e}$  fail to be locally free.  The first of these are the {\it F-regular} rings, which have many direct sum copies of $R$ in $R^{1/p^e}$ as $e$ gets larger:

\begin{defi}\label{def:Freg}\cite{HochsterHunekeTCStrongFRegularity}
An F-finite domain $R$ is  {\bf strongly F-regular}\footnote{For brevity, we often drop the qualifier ``strongly" in the text.  Hochster and Huneke  introduced three flavors of F-regularity---weak F-regularity, F-regularity and strong F-regularity--- and conjectured all to be equivalent. This is still not known in general; however it is known for Gorenstein rings  \cite{HochsterHunekeTCStrongFRegularity}, graded rings 
 \cite{LyubeznikSmithStrongWeakFregularityEquivalentforGraded}, and  even more generally; Cf \cite{AberbachConditionsWeakStrongFRegularity}, \cite{MacCrimmonThesis}. In any case, because our bias is that strong F-regularity is the central notion and these lectures treat only that notion,   we drop the clumsy modifier and frequently write ``F-regular" instead of strongly F-regular, begging forgiveness of Hochster and Huneke. } if for every non-zero element $f \in R$, there exists $e \in \mathbb N$
such that the $R$-module inclusion $Rf^{1/p^e} \hookrightarrow R^{1/p^e}$ splits. Put differently, this means that for all non-zero $f$, there exists
$e \in \mathbb N$ 
 and $\phi \in \Hom_{R}(R^{1/p^e}, R)$ such that $\phi(f^{1/p^e}) = 1$.
\end{defi}

An F-regular ring therefore, may not be free when considered as a module over $R^{p^e}$, but it will have  {\it many } summands isomorphic to $R^{p^e}.$  Indeed, every non-zero element of $R$ will generate an $R^{p^e}$-module direct summand of $R$ for sufficiently large $e$.

\begin{Prop}\label{FregimpliesFsplit}
Every $F$-regular ring is Frobenius split.
\end{Prop}

\begin{proof} Taking $f $ to be $1$, we know that there exists an $e$ such that  $R\hookrightarrow R^{1/p^e}$ splits. Restricting the splitting $\phi: R^{1/p^e}\rightarrow R $ to the subring $R^{1/p},$  we have a splitting of the inclusion $R^{1/p} \hookrightarrow R$. Thus $R$ is Frobenius split.\end{proof}

The proof of the following lemma is a  straight-forward exercise, using the fact that $R^{1/p^e}$ is a finitely generated $R$-module.   
\begin{lemma}\label{Fregprop} \cite[3.1,3.2]{HochsterHunekeTCStrongFRegularity} Let $R$ be an F-finite ring of characteristic $p$.
\begin{enumerate}
\item  $R$ is F-regular  if and only if $R_m$ is F-regular   for every maximal (equivalently, prime) ideal $m $ of $R$. 
\item A  local ring $(R, m)$ is F-regular if and only if the completion $\hat R$ of $R$ at its maximal ideal is F-regular.
\end{enumerate}

\end{lemma}

\begin{Prop}\label{RegimpliesFreg}
An F-finite regular  domain{\footnote{An arbitrary F-finite ring (not a domain) can be defined as strongly F-regular if for every $f$ not in any minimal prime, there exists an $e$ and an $\phi \in Hom_{R}(F^e_*R, R)$ such that $\phi(f) =1$.  However, this is not an essential generalization of the theory:  it is easy to check that an F-regular ring is a product of F-regular domains; see \cite{HochsterHunekeTCStrongFRegularity}.}} is strongly F-regular.
\end{Prop}

\begin{proof}Since F-regularity can be checked locally at each prime (Lemma  \ref{Fregprop}), there is no loss of generality in assuming that $(R, m)$ is local.

The proof is a simple application of Nakayama's Lemma  \cite[Proposition 2.8]{AtiyahMacdonald}.
What does Nakayama's Lemma  say about the finitely generated $R$-module $M = R^{1/p^e}$?
It says that an element $f^{1/p^e}$ is part of a minimal generating set of $R^{1/p^e}$ as an $R$-module if and only if it is not in $mR^{1/p^e}$, which in turn happens if and only if $f $ is not in the ideal $m^{[p^e]}$ in $R$, where 
$m^{[p^e]}$ denotes the ideal of $R$ generated by the $p^e$-th powers of the elements of $m$. 
 Since $\bigcap_e m^{[p^e]} \subset \bigcap_e m^e = 0$,  this means that for each fixed 
 non-zero $f \in R$, we can always find an $e$ such that $f^{1/p^e}$ is a part of a set of minimal generators for $R^{1/p^e}$ over $R$. This observation holds quite generally, whether or not $R$ is regular.

Now if $R$ is  regular, then $R^{1/p^e}$ is a free $R$-module, so that such a minimal generator $f^{1/p^e}$ for $R^{1/p^e}$ over $R$ will necessary be part of a free basis for $R^{1/p^e}$ over $R$. This means that $f^{1/p^e}$ spans a free $R$-module summand of   $R^{1/p^e}$.  Since this holds for every non-zero $f$ (with possibly larger $e$), we conclude that $R$ is F-regular.
\end{proof}

The analog of Proposition \ref{split} holds for F-regular rings, with essentially the same proof:
\begin{Prop}
\label{direct summand of F-regular is F-regular}
Let $R \subset S$ be any inclusion of rings of characteristic $p$ which splits in the category of $R$-modules. If $S$ is F-regular, then $R$ is F-regular. \end{Prop}

 In particular, Veronese subrings of a polynomial ring are F-regular,  as are rings of invariants of finite groups whose order is coprime to the characteristic. Cf Example \ref{ex4}.

The power of  Proposition \ref{direct summand of F-regular is F-regular} stems from the nice properties of F-regular rings: 
\begin{Thm}\label{FregimpliesCM}\cite{HochsterHunekeTCStrongFRegularity} 
F-regular rings are Cohen-Macaulay and normal.
\end{Thm}

For those not already enamored by Cohen-Macaulay singularities, we have included an appendix discussing this crucially important if slightly technical condition.

\begin{proof} 
Without loss of generality, we can assume that $(R, m)$ is an  F-regular local domain. First we  prove that $R$ is Cohen-Macaulay, {\it i.e.} each system of parameters of $R$ is a regular sequence. By Lemma \ref{Fregprop}, we can assume that $R$ is complete, and therefore has a  coefficient field,  $K$. 
By the Cohen Structure Theorem, $R$ is  module finite over the subring $A$ of  formal power series  over $K$ in any system of parameters  $x_1,\dots,x_d$.

Suppose we have a relation on our system of parameters, $x_iz \in (x_1, \dots, x_{i-1})R$. Let $B$ be the intermediate ring generated by $z$ over $A$. 
Note that 
  $B\cong A[t]/(g(t)))$, where $g$ is the minimal polynomial of $z\in B$ over $A$, so that $B$ is a hypersurface ring and in particular, Cohen-Macaulay . To summarize, we have  module finite extensions 
  $$
  A = K[[x_1, \dots, x_d]] \hookrightarrow B  = A[z]    \hookrightarrow R,
  $$ where $B$ is  Cohen-Macaulay with regular sequence
   $x_1, \dots, x_d$. 
  

 Since $B \hookrightarrow R$ is finite, there is a $B$-linear map $\psi:R \rightarrow B$ sending $1$ to, say, $b\neq 0$.  Raising our relation $zx_i \in (x_1, \dots, x_{i-1})R$ to the $p^e$-th power, we have $z^{p^e}x^{p^e}_i\in (x^{p^e}_1,\dots,x^{p^e}_{i-1})R$. 
Applying $\psi$ to this relation, we have   $ bz^{p^e}x^{p^e}_i\in (x^{p^e}_1,\dots,x^{p^e}_{i-1})B
  $
  in $B$.
  But because $B$ is Cohen-Macaulay, we can divide out the $x_i^{p^e}$ to get  $bz^{p^e} \in (x^{p^e}_1,\dots,x^{p^e}_{i-1})B.$ Expanding to $R$, we have equivalently
   \begin{equation}
\label{equation for CM}
   b^{1/p^e}z \in  (x_1,\dots,x_{i-1})R^{1/p^e}.
 \end{equation}
Now using the F-regularity of  $R$, there is an $e$ and $\phi \in \Hom_{R}(R^{1/p^e},R)$ such that $\phi(b^{1/p^e})=1$. Applying $\phi$ to (\ref{equation for CM}) we have $z\in (x_1,\dots,x_{i-1})R$. Thus $R$ is  Cohen-Macaulay.{\footnote{For readers familiar with local cohomology, we leave as an exercise to find a slick local cohomology proof that F-regular rings are Cohen-Macaulay.}}

Next, we tackle normality. Fix an element $x/y$ in the fraction field of $R$ integral over $R$. We must show that $y$ divides $x$ in $R$. Since $x/y$ is integral over $R$, there is an equation
\[(x/y)^m+r_1(x/y)^{m-1}+\cdots+r_m=0,\]
where each $r_j$ is in $R$. Raising both sides of this equation to the $p^e$-th power, we can see that $(x/y)^{p^e}$ is also integral over $R$ for all $e\geq 1$. Since the integral closure $\overline{R}$ of $R$ is module-finite over $R$, there is a $c\in R$ such that $c\overline{R}\subseteq R$; in particular, $c(x/y)^{p^e}\in R$ for all $e\geq 1$, {\it i.e.} $cx^{p^e}\in (y^{p^e})$ for all $e\geq 1$. Hence $cx^{p^e}=ry^{p^e}$ for some $r\in R$. Therefore we have
\begin{equation}
\label{equation for normality}
c^{1/p^e}x=r^{1/p^e}y
\end{equation}
Since $R$ is F-regular, there is an $e$ and $\phi_e\in \Hom_{R}(R^{1/p^e},R)$ such that $\phi_e(c^{1/p^e})=1$. Applying $\phi_e$ to (\ref{equation for normality}) we have that $x\in (y)$. This shows that $x/y\in R$ and finishes the proof.
\end{proof}

To summarize, we now have proved the following implications among classes of singularities:
$$
\{\text{Regular}\} \implies \{\text{F-regular}\} \implies \{ {\text{Frobenius split, Cohen-Macaulay,  Normal}}\}.
$$
In addition, we have shown that both Frobenius splitting and F-regularity descend to direct summands.  This is all that is needed to prove the Hochster-Roberts theorem, at least in characteristic $p$.

\subsection{The Hochster-Roberts Theorem.} \label{prehistory} 

\begin{Thm}[1974]
\label{HR} Fix any ground field $k$. 
Let $G$ be a linearly reductive algebraic group  over $k$ acting on a regular Noetherian $k$-algebra  $S$. Then the ring of invariants
$$
S^G := \{ f \in S \,\, | \,\, f \circ g = f \,\, {\rm { \,\,\,for \,\, all\,\,\, }} g \in G\};
$$
 is Cohen-Macaulay.
\end{Thm}

For example, let $V$ be a finite dimensional representation of a linearly reductive group $G$. 
The Hochster-Roberts theorem guarantees that the ring of invariants for the induced action of  $G$ on the symmetric algebra  of $V$ is Cohen-Macaulay. From a practical point of view, this means that 
the invariants form  a graded finitely generated {\it free\/}  module over some polynomial subring.

Geometrically,  the point of the Hochster-Roberts theorem is that when a reasonable group acts on a smooth variety, the ``quotient variety" will  have reasonably nice singularities.  Indeed, let $X$ be a smooth (affine) variety on which the group $G$ acts by regular maps. Then there is an induced action on the coordinate ring $S$, and more or less by definition, the  ``quotient variety" is the unique variety $X/G$ whose coordinate ring is $S^G$.  The Cohen-Macaulayness of $S^G$  is a niceness condition on the singularities of the quotient.{\footnote{Some caution is in order here: one can not usually put the structure of a variety on the set of $G$-orbits of $X$, although in a sense that can be made precise, Spec $S^G$ is the algebraic variety ``closest to being  a quotient"---   it behaves as a quotient in a categorical sense. If  $G$ is finite (and its order is not divisible by $p$),  the topological space Spec $ S^G$ {is}  the topological quotient of Spec $S$ by $G$.
In the non-affine case, the situation is even more complicated, and  there are several ``quotients," which depend  on a choice of {\it linearization\/} of the action. This is the huge and beautiful theory of  geometric invariant theory, or GIT \cite{MumfordGITBook}. }}

\subsubsection {Linear Reductive Groups.} By definition, a {\bf linearly reductive group} is an algebraic group with the property that every finite dimensional representation is completely reducible, that is, it decomposes as  a direct sum of irreducible representations. In characteristic zero, linearly reductive is the same as reductive, so includes all semi-simple algebraic groups. In particular, 
all finite groups, all tori, and all matrix groups such as  $ GL_n $ and   $SL_n$ are  linearly reductive over  a field of  characteristic zero. Over a field of positive characteristic $p$, linearly reductive groups are less abundant: tori and finite groups whose order is not divisible by $p$, as well as extensions of these.  See \cite{NagataCompleteReducibility} for more information.

The point for us is this: if $G$ is a linearly reductive group acting on a regular $k$-algebra $S$,
then  the inclusion of the ring of invariants $S^G$ in $S$ splits.

 Using this, it is easy to prove the Hochster-Roberts' theorem in the prime characteristic case.
We refer to the original paper \cite{HochsterRobertsInvSubring} for the reduction to prime characteristic.

\begin{proof}[Proof of the Hochster-Roberts' Theorem in Prime Characteristic.]
Because $S^G$ is a  {\it direct summand\/} of the ring $S$, the Hochster-Roberts theorem  follows immediately from 
the following:
\begin{Thm}
Let $R \subset S$ be a split inclusion of rings of positive characteristic. If $S$ is regular, then $R$ is F-regular, hence Cohen-Macaulay.
\end{Thm}
This theorem in turn is simply a stringing together of Proposition \ref{RegimpliesFreg}, which tells us $S$ is F-regular, Proposition \ref{direct summand of F-regular is F-regular}, which guarantees that the property of F-regularity is passed on to the direct summand $S^G$, and finally Proposition \ref{FregimpliesCM}, which tells us that $S^G$ is therefore Cohen-Macaulay and normal.
\end{proof}

While the Hochster-Roberts theorem  is most interesting  in characteristic zero (since that is where the most interesting groups are found), its original proof fundamentally uses  characteristic $p$.
 Later, Boutot gave a different proof of the characteristic zero case, which does not use reduction to prime characteristic, although it  still exploits the philosophy of ``splitting" we discuss here  \cite{BoutotSummabdRationalSing}.  This philosophy is further expounded in \cite{KovacsCharacterizationRatSing}.

\begin{ex}
\label{ex1}
Let $G $ be the two-element group $\{\pm 1\}$ under multiplication. Let $G$ act on $k^2,$ where $k$ is a field whose characteristic is not two,  in the obvious way  by multiplication:
$-1 \cdot (x, y) = (-x, -y)$.  The induced action on the coordinate ring $k[x, y]$ is the same: $-1$ acts by multiplying both $x$ and $y$ by $-1$, so that a monomial $x^ay^b$ is sent to $(-1)^{a+b} x^a y^b$. In particular, the invariant ring is the subring generated by polynomials of even degree, or
$$
 k[x, y]^G =  k[x^2, xy, y^2] \cong   k[u, v, w]/(v^2 - uw).
 $$
 According to the Hochster-Roberts theorem and its proof, this ring is F-regular in all finite characteristics, and hence Cohen-Macaulay. Standard ``reduction to characteristic $p$ techniques" guarantee  the ring is Cohen-Macaulay also when $k$ has characteristic zero. 
\end{ex}

\begin{ex}\label{ex2}
Let $X$ be the variety of all $2 \times n$ complex matrices, so $X \cong \mathbb C^{2n}$. Let $G = SL_2$ act on $X$ by left multiplication. The ring of invariants for the induced action of $SL_2$ on
$\mathbb C[x_{11}, \dots, x_{1n}, x_{21}, \dots, x_{2n}]$ is generated by all the $2 \times 2$ subdeterminants of the matrix of indeterminates:
$$
\left(
\begin{array}{cccc}
 x_{11} & x_{12}  & \dots & x_{1n}   \\
 x_{21} &  x_{22}  & \dots & x_{2n}
 \end{array}
\right).
$$ This is the homogeneous coordinate ring for the Pl\"ucker embedding of the Grassmannian of 2-dimensional subspaces of $\mathbb C^n$.
The  Hochster-Roberts Theorem guarantees that this ring is Cohen-Macaulay.
More generally, the Pl\"ucker ring of the  Grassmannian of $d$-dimensional subspaces of $\mathbb C^n$
is Cohen-Macaulay, since it is the ring of invariants for $SL_d$ acting on $\mathbb C^{d\times n}$ by left multiplication. In characteristic $p$, the group $SL$ is not linearly reductive and the ring of invariants does not split. None-the-less, this ring is F-regular in all prime characteristics\cite{HochsterHunekeParameterIdealSplitting}, hence Cohen-Macaulay.
\end{ex}

\subsection{F-signature.}
A numerical refinement of F-regularity called the {\it F-signature}  sharpens the classification of F-singularities  by measuring the {\it growth rate\/} of the rank of a maximal free summand of the $R$-module $R^{1/p^e}$ as $e$ goes to infinity.  This was first studied in \cite{SmithVandenBergh}. 

Fix a local F-finite domain $R$. For each natural number $e$, we can decompose the 
$R$ module $R^{1/p^e}$ as a direct sum of  indecomposable modules, and count the number $a_e$ of summands  of $R^{1/p^e}$ that are isomorphic to $R$. For a Frobenius split ring $R$, there is always at least one. If $R$ is regular, all summands are isomorphic to $R$,  so $a_e$  is equal to the rank of $R^{1/p^e}$ over $R$. For arbitrary $R$,  the (generic) rank of $R^{1/p^e}$ over $R$ 
obviously bounds the number $a_e$.   For an F-regular ring, we expect many summands of $R^{1/p^e}$ isomorphic to $R$, so we expect $a_e$ to be relatively large, and to grow with $e$. In fact, as it turns out, we can define the F-signature of $R$ to be
$$
s(R) = \lim_{e\rightarrow \infty} \frac{a_e}{\delta^e},
$$
where $\delta $ is the generic rank of $R$ over $R^p$, 
 that is $\delta= [K:K^p]$ where $K$ is the fraction field of $R$.  This limit exists \cite{TuckerFSigExists}, and is at most one. 

The F-signature can be used to classify F-regular rings. Indeed, 
Huneke and Leuschke proved that the F-signature is one if and only if $R$ is regular  in the paper \cite{HunekeLeuschkeMCMModules} that coined the term ``F-signature." Furthermore, 
the  F-signature is positive if and only if $R$ is F-regular \cite{Aberbach-Leuschke}. Thus each F-regular ring has an F-signature strictly between zero and one; 
the closer the F-signature is to one, the ``less singular" the ring is.  For example,  for rational double points such $xy = z^{n+1}$, the F-signature is $1/(n+1)$ \cite{HunekeLeuschkeMCMModules}, reflecting the fact that the singularity is ``worse" for larger $n$. Many more computations of this type can be found in  \cite{HunekeLeuschkeMCMModules} and later \cite{Yao06}.  Formulas for the F-signature of toric varieties are worked out in Von Korff's PhD thesis \cite{VonKorffThesis}; see also \cite{Watanabe-Yoshida} and \cite{SinghFSigSemiGroupRing}.  Tucker vastly generalizes and simplifies much of the literature on F-signature  in \cite{TuckerFSigExists}. 

There are many interesting open questions about the F-signature. No known examples of non-rational F-signatures are known (though some expect that they exist). Also, Florian Ensecu has suggested that  there may be  an upper bound on the $F$-signature of a non-regular ring depending only on the dimension: for $d\geq 2$, the singularity defined by $x_0^2 + x_1^2 + \dots + x_{d}^2$ (characteristic $\neq 2$) has F-signature $1 - \frac{1}{d}$, and no $d$-dimensional singularities of larger F-signature are known \cite{FlorianPersonalCommunication}.  The F-signature is closely related to the Hilbert-Kunz multiplicity, a subject pioneered by Paul Monsky \cite{MonskyHKFunction}; see Huneke's survey \cite{HunekeLectureHKFSignature} or Brenner's paper \cite{BrennerIrratHKMultip}.
   For more  developments,  see the recent papers of Blickle, Schwede and Tucker \cite{BlickleSchwedeTuckerFSigFSplitting}  and \cite{BlickleSchwedeTuckerpfractals}, which include generalizations of F-signature to pairs.

\subsection{Frobenius splitting in characteristic zero and Connections with Singularities in Birational Geometry.}\label{char0stuff}
We briefly recall the standard technique for extending these ideas to algebras over fields of characteristic zero. Let $\mathbb C$ denote any field of characteristic 0.

 Let $R$ be a finitely generated  $\mathbb C$-algebra. Fix a presentation
  $$R\cong \mathbb{C}[x_1,\dots,x_n]/(f_1,\dots,f_r).$$ Let $A$ be the $\mathbb{Z}$-subalgebra of $\mathbb C$ generated by all  coefficients of  the polynomials $f_1,\dots,f_r$, and set 
 $$R_A=A[x_1,\dots,x_n]/(f_1,\dots,f_r).$$ Since $A$ is a finitely generated  $\mathbb{Z}$-algebra, the residue field of $A$ at each of its maximal ideals is finite. The map 
 Spec $R_A \rightarrow $ Spec $A$ can be viewed as a ``family of models" of the original algebra $R$. The closed fibers  of this map are  characteristic $p$ schemes (of varying $p$) whereas the generic fiber is a flat base change from the original $R$. Roughly speaking, $R$ is F-regular or Frobenius split if most (or at least a dense set) of the closed models have this property. More precisely, 

\begin{defi}\label{cha0}
Let $R, A,R_A$ be as above. The ring $R$ is said to have {\it  Frobenius split type} (or {\it  F-regular type}) if there is a Zariski dense set of  maximal ideals $\mu$ in Spec $A$ such that $A/\mu\otimes_AR_A$ is Frobenius split (or F-regular).{\footnote{In the literature, this is usually called {\it dense\/} F-split (F-regular) type. The related condition that  $A/\mu\otimes_AR_A$ is F-split  (F-regular) for a Zariski {\it open\/} set of maximal ideal $\mu$ in Spec $A$ is called {\it open} F-split (F-regular)  type. }} 
\end{defi}

Although it is not completely obvious, Definition \ref{cha0} does not depend on the presentation of $R$, nor on the choice of $A$. See \cite{HochsterHunekeTightClosureInEqualCharactersticZero}.
\begin{ex}\label{examples-type}
\begin{enumerate}
\item The ring $\C[x,y,z]/(y^2-xz)$ has  F-regular type. As a matter of fact, taking $A = \mathbb Z$,  the closed fibers of the family are the rings $\mathbb F_p[x,y,z]/(y^2-xz),$  which are 
  F-regular for {\it every\/}  prime number $p$. Cf. Example \ref{ex4} and Proposition \ref{direct summand of F-regular is F-regular}.
\item The ring $\C[x,y,z]/(x^3+y^3+z^3)$ has Frobenius split type, but {\it not} F-regular type. Indeed, 
$\mathbb F_p[x,y,z]/(x^3+y^3+z^3)$ can be  checked to be Frobenius split after reduction mod $p$ whenever $p\equiv 1$ mod $3$, so 
that there is an infinite  set of prime numbers $p$ (hence dense set of Spec $\mathbb Z$) for which 
the ``reduction mod $p$" is Frobenius split.{\footnote{If $p \equiv 2 \, $mod $3,$ the ring is not Frobenius split, so $R$ has {\it dense \/} F-split type, but not {\it open \/} F-split type. Cf. Example \ref{elliptic}. On the other hand, it is expected that dense and open F-regular type are equivalent; this is known in Gorenstein rings and related settings; see \cite{SmithRationalSingularities} and \cite{HaraWatanabe}.}} On the other hand,  for every $p \geq 5$ and every $e$,  one can show that there is {\bf no  map\/}  sending $x^{1/p^e}$ to 1. So this ring is not F-regular type. \end{enumerate}
\end{ex}

\bigskip
\subsubsection{Connections with the Singularities in the Minimal Model Program.}\label{MMPsing}
Amazingly, the properties of Frobenius splitting and F-regularity in characteristic zero turn out to be closely related to a number of important issues studied independently in algebraic geometry, including log canonical and log terminal singularities, and ultimately positivity and multiplier ideals as well. 
For example, 
\begin{Thm}\label{mmp}  \cite{SmithRationalSingularities}, \cite{HaraCharacterizationRatSing}  \cite{MehtaSrinivasRatImpliesFRat}, \cite{ElkikRatCanSing}
Let $R$ be a Gorenstein domain finitely generated over a field of characteristic zero. 
Then 
\begin{enumerate}
\item $R$ has F-regular type if and only if $R$ has log terminal singularities. 
\item If $R$ has Frobenius split type, then $R$ has log canonical singularities. 
\end{enumerate}
\end{Thm}

 We do not digress to discuss all the relevant definitions here, but refer instead to the literature. For Gorenstein varieties, log terminal is equivalent to rational singularities \cite{ElkikRatCanSing}, which may be more familiar.
 
 Theorem \ref{mmp} is proven more generally by Hara and Watanabe in \cite{HaraWatanabe}.  Statement (1) is closely related to the equivalence of rational singularities with F-rational type (proved by Smith \cite{SmithRationalSingularities} and Hara/Mehta-Srinivas \cite{HaraCharacterizationRatSing}  \cite{MehtaSrinivasRatImpliesFRat}).   
 Statement (2)  is closely related to the fact F-injective type implies Du Bois singularities (due to Schwede \cite{SchwedeFInjectiveAreDuBois}). 
 
 The converse of 
 Statement (2) is conjectured to hold in general as well. 
 This long-standing open question  related to an important conjecture linking the {\it F-pure threshold}  and  the {\it log canonical threshold, }  a very rich area of research with a huge literature.
 While still wide open, it is worth pointing out that Mustata and Srinivas \cite{MustataSrinivas}  have reduced the question to an interesting conjecture with roots in Example \ref{elliptic}. 
 Although we cannot go into the F-pure threshold (or it generalization to the ``F-jumping numbers") here,   fortunately there are already extensive recent surveys, including the survey for beginners \cite{BenitoFaberSmith} and the more advanced research survey \cite{MustataIMPANGA}.  The F-pure threshold is very difficult to compute, often with complicated
fractal-like behavior; see, for example, the  
 papers \cite{HernandezFPTBinomial}, \cite{HernandezTexiera} and \cite{HernandezFPTDiagonal} for concrete computations of F-thresholds. There are few general results, but some can be found in \cite{BhattSingh} and \cite{HernandezNunezBetancourtWittZhang}.

\newpage
\section{Frobenius for Schemes: the Global Theory}

Let $X$  denote the affine scheme  $\Spec R$, where $R$ is a ring of prime characteristic $p$.  Like any map between rings, the Frobenius map induces a map of schemes, which we also denote $F$. As a map of the underlying topological space, $F: X \rightarrow X$  is the  {\it identity\/}  map,  but the associated  map of sheaves  $\mathcal O_X \rightarrow F_*\mathcal O_X$ is induced by the $p$-th power  map of $R$.

Of course,  the $p$-th power map is compatible with localization, so that  the Frobenius map on affine charts can be patched together to get a Frobenius map for  {\it any scheme $X$ of characteristic $p$.\/}  This  Frobenius map is the identity map on the underlying topological space of $X$ while the corresponding map of  sheaves of rings $\mathcal O_X \rightarrow F_*\mathcal O_X$ is the $p$-th power map locally on sections.

The sheaf $F_*\mathcal O_X$ is  quasi-coherent on $X$. Consistent with the terminology for rings, we say that a scheme $X$ is {\it F-finite\/} if the sheaf $F_*\mathcal O_X$ is {\it coherent.}  Our main interest is when $X$ is a variety over a perfect field $k$ of characteristic $p$; such a variety is always F-finite.\footnote{The Frobenius map is always a {\it scheme map,\/} but  not usually a morphism of varieties over $k$, because it is not linear over $k$ (unless, for example, $k = \mathbb F_p$). If we insist on working with maps of varieties, we can force the Frobenius map to be defined over $k$ by changing base to make this so; this is called the relative Frobenius map.  See e.g. \cite{Mehta-Ramanathan} \cite{BrionKumar}.}

 As in the affine case,  the sheaf  $F_*\mathcal O_X$ carries a remarkable amount of information about the scheme $X$. For example,   Theorem \ref{kunzthm}  implies that 
an F-finite scheme  $X$  is regular if and only if the coherent $\mathcal O_X$-module $F_*\mathcal O_X$ is locally free. That is, a variety over a perfect field is smooth if and only if the coherent sheaf  $F_*\mathcal O_X$  is a {\it vector bundle} over $X$.

 Similarly, we 
 can define a scheme $X$ to be {\it locally Frobenius split} if the map $\mathcal O_X \rightarrow F_*\mathcal O_X$ splits locally in a neighborhood of each point, or equivalently, if the corresponding map on stalks splits for each $p \in X$. Likewise,  we can define $X$ to be {\it locally F-regular}  if the stalks are all  F-regular. 
 Since  Frobenius splitting and F-regularity are {\it local\/} properties for  affine schemes  (by Lemmas \ref{loc} and \ref{Fregprop}), all the results from the previous section give corresponding local results for an arbitrary F-finite scheme of prime characteristic. For example, a locally F-regular scheme is  normal and Cohen-Macaulay by Theorem \ref{FregimpliesCM}. 

It is much stronger, of course, to require a {\it global}  splitting of the Frobenius map   $\mathcal O_X \rightarrow F_*\mathcal O_X$. Not surprisingly, a global splitting of Frobenius has  strong consequences for the global geometry of $X$.  This is the topic of our second lecture.

\begin{defi} \label{defFsplit}The scheme
$X$ is  {\it Frobenius split}{\footnote{globally Frobenius split, if there is any possibility of confusion}}  if the Frobenius map $\mathcal{O}_X\to F_*\mathcal{O}_X$ splits as a map of $\mathcal{O}_X$-modules. This means that there exists a map of sheaves of $\mathcal O_X$-modules, $F_*\mathcal O_X \overset{\phi}\longrightarrow \mathcal O_X$ such that the composition
$$
\mathcal{O}_X\to F_*\mathcal{O}_X  \overset{\phi}\longrightarrow \mathcal O_X$$
is the identity map.
\end{defi}

The global consequences of splitting Frobenius, and indeed the term {\it Frobenius split,\/} were first treated systematically by Mehta and Ramanathan in \cite{Mehta-Ramanathan}; see also \cite{HaboushShortProofKempfVanishing}.  While inspired 
  by Hochster and Roberts' paper ten years prior, which focused on the local case, Mehta and Ramanathan  were motivated by the possibility of understanding the {\it global } geometry of Schubert varieties and related objects in algebro-geometric representation theory; see {\it e.g.} \cite{Mehta-Ramanathan}, or \cite{RamananRamanathanProjNormality}, \cite{MehtaRamanathanSchubertVar}.
  This idea was very fruitful, leading the Indian school of algebro-geometric representation theory to many important results   now chronicled in the book \cite{BrionKumar}.
   In Section \ref{locvsglob},  we formally show how the  local and global  points of view converge by translating global splittings of  a projective variety $X$ into local splittings ``at the vertex of the cone" over $X$.

\begin{ex} Projective space $\mathbb P^n_k$ is Frobenius split in every positive characteristic. Indeed, 
any (homogeneous) splitting of Frobenius for the polynomial ring $k[x_0, \dots, x_n]$ induces a splitting of the corresponding Frobenius map of sheaves $\mathcal O_{\mathbb P^n} \rightarrow F_*\mathcal O_{\mathbb P^n}$.
\end{ex}

  Frobenius split varieties satisfy strong vanishing theorems: 
\begin{Thm}\label{vanish}
Let $X$ be a Frobenius split scheme. If  $\mathcal{L}$ is an invertible sheaf on $X$ such that $H^i(X,\mathcal{L}^n)=0$ for $n\gg 0$, then $H^i(X,\mathcal{L})=0$.
\end{Thm}\label{vanishing}
\begin{Cor}
\label{ample on F-split implies vanishing}
Let $\mathcal L$ be a ample invertible sheaf on a Frobenius split projective variety $X$. Then $H^i(X,\mathcal{L}) $  vanishes for all $i > 0$, and  if $X$ is Cohen-Macaulay ({\it e.g.} smooth),  then
also $H^i(X,\omega_X\otimes\mathcal{L})$ vanishes 
for all $i>0$.
\end{Cor}

\begin{proof}[Proof  of Theorem \ref{vanish} and its Corollaries]
By definition, the map $\mathcal{O}_X\to F^e_*\mathcal{O}_X$ splits. So tensoring with $\mathcal L$,  we also have a splitting of 
\[\mathcal{L}\to F^e_*\mathcal{O}_X\otimes \mathcal{L}=F^e_*F^{e*}\mathcal{L}=F^e_*\mathcal{L}^{p^e}.\] 
Here, the first equality follows from the projection formula  \cite[Exercise II 5.2 (d)]{Hartshorne};  the second equality $F^{e*}\mathcal L = \mathcal L^{p^e}$ holds because pulling back under Frobenius raises transition functions to the $p$-th power.  Since $\mathcal L$ is a direct summand of $F^e_*\mathcal{L}^{p^e},$   it follows that 
 the cohomology $H^i(X,\mathcal{L})$ is a direct summand of $H^i(X,F^e_*\mathcal{L}^{p^e})=H^i(X,\mathcal{L}^{p^e})$ for all $e$. But $H^i(X,\mathcal{L}^{p^e})=0$ for large $e$ by our hypothesis, so that $H^i(X,\mathcal{L})=0$.
 
 The corollary follows by Serre vanishing  \cite[Prop III 5.3]{Hartshorne}: an ample invertible sheaf $\mathcal L$ on a projective variety $X$ satisfies $H^i(X, \mathcal  L^n) = 0$ for large $n$ and positive $i$. The second statement follows by Serre duality \cite[III \S 7]{Hartshorne}: 
 Serre vanishing ensures $H^i(X, \omega_X \otimes \mathcal L^n)$ vanishes for large $n$, hence its dual $H^{\dim X-i}(X,  \mathcal L^{-n})$ vanishes; by the Theorem  also  $H^{\dim X-i}(X,  \mathcal L^{-1})$  vanishes, and hence so does its dual  $H^i(X, \omega_X \otimes \mathcal L)$. 
\end{proof}

Proving  a  variety is Frobenius split is therefore a worthwhile endeavor. One useful
criterion is essentially due to Hochster and Roberts: 
\begin{Prop}\label{Fcrit}
 A projective variety $X$ is Frobenius split if and only if
the induced map  $H^{\dim X}(X, \omega_X)  \rightarrow H^{\dim X}(X, F^*\omega_X) $ is injective. 
\end{Prop}

\begin{proof} Hochster and Roberts actually stated a local version of this proposition:  to check 
purity of any finite map of algebras $R \rightarrow S$, where $R$ is local,  it is enough to show that  the map $R \otimes E \rightarrow S \otimes E$ remains injective after tensoring with the injective hull  $E$ of the residue field of $R$. This statement reduces to the statement of Proposition \ref{Fcrit} by taking $R$ to be the localization of a homogeneous coordinate ring of $R$ at the unique homogeneous maximal ideal. See \cite[4.10.2]{SmithSantaCruz}. 
\end{proof}

\begin{ex}\label{elliptic} An elliptic curve over a perfect field of prime characteristic is Frobenius split if and only if it is {\it ordinary.} Indeed, ordinary means that the natural map induced by Frobenius $H^1(X, \mathcal O_X)  \rightarrow H^{1}(X, F^*\mathcal O_X)$ is injective, so the statement follows from Proposition \ref{Fcrit} since the canonical bundle of an elliptic curve is trivial. Thus in genus one, there are infinitely many Frobenius split curves, as well as infinitely many non-Frobenius split curves, depending on the Hasse-invariant of the curve. See, for example, \cite[IV \S 4 Exercise 4.14]{Hartshorne}.  Cf. Example \ref{examples-type} (2).
\end{ex}

\subsection{Global F-regularity.}
We  define a global analog of F-regularity for arbitrary integral F-finite schemes of prime characteristic $p$. For any effective Weil divisor $D$ on a normal variety, there is an obvious inclusion $\mathcal O_X \hookrightarrow \mathcal O_X(D)$. Thus for any $e$ we have an inclusion $F^e_*\mathcal{O}_X\hookrightarrow F^e_*\mathcal{O}_X(D)$, which we can pre-compose with the iterated Frobenius map to get a map 
\[\mathcal{O}_X\to F^e_*\mathcal{O}_X\hookrightarrow F^e_*\mathcal{O}_X(D).\]

\begin{defi}\label{defFreg} \cite{SmithGloballyF-Regular} An F-finite normal{\footnote{If $X$ is quasi projective, we can drop the normal from the definition and assume instead that $D$ is an effective Cartier divisor. This produces the same definition, because splitting along all Cartier divisors will imply normality (by Theorem \ref{FregimpliesCM}) as well as  splitting along all Weil divisors (since given an effective Weil divisor $D$, we can always find an effective Cartier divisor $D'$ such that $\mathcal O_X(D) \subset \mathcal O_X(D')$.))}}
  scheme $X$ is called {\it globally F-regular} if, for all effective Weil divisors $D$, there is an $e$ such that the composition
\[\mathcal{O}_X\to F^e_*\mathcal{O}_X\hookrightarrow F^e_*\mathcal{O}_X(D)\]
splits as a map of $\mathcal{O}_X$-modules. \end{defi}

Globally F-regular varieties are  strongly Frobenius split, in the sense  that not only are they Frobenius split but there are typically {\it many\/} splittings of Frobenius.  Indeed, suppose that $X$ is globally F-regular.  Then for {\it  any} effective divisor $D$, 
 there is an $e$  and a map $F^e_*\mathcal{O}_X (D) \overset{\phi}\rightarrow \mathcal O_X$ such that the composition  
\[\mathcal{O}_X\to F^e_*\mathcal{O}_X  \overset{\iota}\hookrightarrow 
 F^e_*\mathcal{O}_X(D) \overset{\phi}\rightarrow \mathcal O_X
\]
is the identity map. 
Thus we can view  the composition $\phi \circ \iota$  as a splitting of the (iterated) Frobenius $\mathcal O_X \rightarrow F^e_*\mathcal O_X$, which just happens to factor through  $F^e_*\mathcal O_X(D)$. 
Thus there are actually {\it many} splittings of the (iterated) Frobenius, as these are typically different maps for different $D.$  Of course, the Frobenius itself splits as well (not just its iterates):  we have a factorization
 $$\mathcal{O}_X \hookrightarrow F_*\mathcal{O}_X \hookrightarrow  F^e_*\mathcal{O}_X,$$  so the splitting can be restricted to  $ F_*\mathcal{O}_X$.
 
 Thus we have proved the following analog of Proposition \ref{FregimpliesFsplit}:
 \begin{Prop}
A  globally  F-regular scheme $X$ is always Frobenius split.
\end{Prop}

 Thus a globally F-regular variety can be viewed as belonging to a restricted class of Frobenius split varieties, in which there are many different splittings of (the iterated)  Frobenius---indeed so many that, for {\it every}  Weil divisor on $X$, we can find a splitting factoring through $F^e_*\mathcal{O}_X(D)$. 
 Such a Frobenius splitting is said to be a {\it Frobenius splitting along the divisor $D$} \cite{RamanathanFSplittingandSchubertVar} \cite{RamananRamanathanProjNormality}.

\begin{Rmk}
The reader will easily verify that 
for affine schemes, the local and global definitions  of F-regularity  are equivalent.  Let us point out only this much:  given an effective Cartier divisor $D$, we can chose a sufficiently small affine chart so that $D$ has 
 local defining equation $f$  on Spec $R$. This means that 
 $\mathcal{O}_X(D)$ is the (sheaf corresponding to the) invertible $R$-module $R\frac{1}{f}$. 
Then the map $
\mathcal{O}_X\to F^e_*\mathcal{O}_X(D)
$
of Definition \ref{defFreg} 
corresponds to the $R$-module map
$$
R \rightarrow [R\cdot \frac{1}{f}]^{1/p^e},\,\,\, {\rm{sending }} \,\,\,  1 \mapsto 1,$$
which splits if and only if the map $R \rightarrow R^{1/p^e}$ sending $1 \mapsto f^{1/p^e}$  splits. 
This is Definition \ref{def:Freg}.
 \end{Rmk}

For a non-affine scheme $X$, the global splitting of the map  
$ \mathcal{O}_X\to  F^e_*\mathcal{O}_X(D)$ 
is a {\it strong condition, which can not be checked  locally at stalks.} Thus $X$ can fail to 
be globally F-regular, and usually does, even when it is locally F-regular at each point:   globally F-regular varieties are rare even  among smooth projective varieties. For example, all smooth projective curves are locally F-regular (because all local rings are regular), but the only smooth projective curve which is globally F-regular is $\mathbb P^1$. This follows immediately from the vanishing theorem:

\begin{Cor}\label{nef} If $\mathcal L$ is a nef{\,}{\footnote{By definition, an invertible sheaf on a curve is nef if it has non-negative degree;  an invertible sheaf on a higher dimensional variety  
is nef if its restriction to every algebraic curve in the variety is nef. Ample line bundles are always nef.  Nef line bundles play an important role in higher dimensional birational geometry, being the ``limits of ample divisors". See Section 1.4 of  \cite{LazPAG1}. }} invertible sheaf on a globally F-regular projective variety, then   $H^i(X,\mathcal{L})$ vanishes for all  $\mathcal{L}$ and all $i>0$.
\end{Cor}

Now to see that 
$\mathbb P^1$ is the only globally F-regular projective
curve, note that the degree of  the canonical bundle on a curve of genus $g$ is  $2g-2$ so the canonical bundle is nef when the genus is positive. But 
since   $H^1(X, \omega_X) = 1$ for all connected curves,
Corollary \ref{nef} prohibits a curve of positive genus from being globally F-regular.

\begin{proof}  The proof is similar to  the proof of Theorem \ref{vanish}, so we only sketch it, referring to \cite[Thm 4.2] {SmithGloballyFReg} for details. For any effective $D$, we can use the splitting of a map $\mathcal O_X \rightarrow F^e_* \mathcal O_X(D)$ to show that  if $\mathcal L$ is any invertible sheaf such that $H^i(X, \mathcal L^n\otimes \mathcal O_X(D))$ vanishes for all large $n$ and some effective $D$, then also $H^i(X, \mathcal L)$ vanishes. Now the corollary follows  because if $\mathcal L$ is nef, there is an effective $D$ such that all $ \mathcal L^n\otimes \mathcal O_X(D)$ are ample by \cite[Cor 1.4.10]{LazPAG1}.
\end{proof}



In practice, we do not have to check splitting for {\it all } divisors $D$  to establish global F-regularity:

\begin{Thm}\label{testelement}
A projective variety $X$ is globally F-regular if for some ample divisor $D'$ containing the singular locus of $X$ and all divisors $D$ of the form $mD'$ for $m \gg 0$, there is an $e$ such that the map $\mathcal O_X \rightarrow F_*^e\mathcal O_X(D)$ splits as a map of $\mathcal O_X$-modules.  \end{Thm}

\begin{proof}
This is a global adaptation, proved in \cite{SmithGloballyFReg}, of the following local theorem of Hochster and Huneke: If $c$ is an element of an F-finite domain  $R$ such that $R_c$ is regular, then $c$ has some power $f$ such that $R$ is F-regular if and only if there exists an $e$ and an $R$-linear  
map $R^{1/p^e} \rightarrow R$ sending $f^{1/p^e} \mapsto 1.$ 
More general and more effective versions of this theorem are proved in \cite{SmithGloballyFReg} and \cite[Thm 3.9]{SchwedeSmithLogFanoVsGloballyFRegular}. 
\end{proof}

\subsection{Local versus Global splitting.}\label{locvsglob} The precise relationship between local and global Frobenius splitting is clarified by the following theorem, which states roughly that 
  a projective variety  is Frobenius split  or globally F-regular if and only if  ``its affine cone" has that property: 
\begin{Thm}
\label{criterion F-split1}
Let $X \subset \mathbb P^n$ be a normally embedded projective variety over a field of characteristic $p$. Then $X$ is Frobenius split (respectively, globally F-regular)  if and only if the corresponding  homogenous coordinate ring  is Frobenius split (respectively, globally F-regular). 
\end{Thm}

\begin{ex} 
Grassmannian varieties  of any dimensions and characteristic are globally F-regular.  Indeed,  the homogeneous coordinate ring for the Pl\"ucker embedding of any Grassmannian is F-regular \cite{HochsterHunekeParameterIdealSplitting}. More generally, all Schubert varieties are globally F-regular \cite{LauritzenRabenThomsenGlobalFRegularityOfSchuertVarieties}.  C.f. Example \ref{ex2}.
\end{ex}

\begin{ex} A normal projective toric variety (of any characteristic) is  globally F-regular. 
The point is that there is a torus-invariant  ample divisor,  so some multiple of it will give a normal embedding into projective space. The
corresponding homogenous coordinate ring  is  a normal domain generated by monomials. All finitely generated normal rings generated by monomials are strongly F-regular since they are direct summands of the corresponding polynomial ring \cite[Exercise 6.1.10]{BrunsHerzogCMrings}. Since the homogenous coordinate ring is F-regular, the projective toric variety is globally F-regular by Theorem \ref{criterion F-split}.  
See also Example \ref{toric}.
\end{ex}

The proof of Theorem \ref{criterion F-split1} is clearest in the following context, which is only slightly more general.
 \begin{Thm} \label{criterion F-split} 
 Let $X$ be any projective  scheme over a perfect field.  The following are equivalent:
\begin{enumerate}
\item $X$ is Frobenius split;
\item the ring $S_{\mathcal{L}}=\bigoplus_{n\in \mathbb{N}}H^0(X,\mathcal{L}^n)$ is Frobenius split for all invertible sheaves $\mathcal{L}$;
\item the section ring $S_{\mathcal{L}}=\bigoplus_{n\in \mathbb{N}}H^0(X,\mathcal{L}^n)$ is Frobenius split for some ample invertible sheaf $\mathcal{L}$.
\end{enumerate}
Likewise, 
a projective variety $X$ is globally F-regular if and only if some (equivalently, every) section ring $S_{\mathcal{L}}$ with respect to an ample invertible sheaf $\mathcal{L}$ is F-regular.
\end{Thm}

\begin{proof}[Proof of Theorem \ref{criterion F-split}]
If $\mathcal{O}_X\to F_*\mathcal{O}_X$ splits, then the same is true after tensoring with any invertible sheaf $\mathcal L$. So as in the proof of Theorem \ref{vanish}, 
\[\mathcal{L}\to F_*\mathcal{O}_X\otimes \mathcal{L}=F_*F^{*}\mathcal{L}=F_*\mathcal{L}^{p},\] 
splits.   Likewise, we have a splitting after tensoring with the sheaf of algebras $\bigoplus_{n\in\mathbb N} \mathcal L^n. $
Taking global sections produces a Frobenius splitting for  the ring $S_{\mathcal L}$. So (1) implies (2). Also (2) obviously implies (3).

To see that (3) implies (1), we fix an ample invertible sheaf $\mathcal L$ on $X$. In particular, this means  that   $X $ is the scheme Proj $S_{\mathcal L}$, and coherent sheaves on $X$ correspond to finitely generated $\mathbb Z$-graded $S_{\mathcal L}$-graded modules (up to agreement in large degree).   Now, if $S_{\mathcal L}$ is Frobenius split, then we can find a homogeneous $S_{\mathcal L}$-linear splitting $S_{\mathcal L}^{1/p} \overset{\pi}{\rightarrow} S_{\mathcal L}$ such that the composition
\begin{equation}\label{eq1}
S \hookrightarrow S_{\mathcal L}^{1/p} \overset{\pi}\rightarrow S_{\mathcal L}
\end{equation} 
is the identity map. 
Note that $S_{\mathcal L}^{1/p}$ can be viewed as naturally $\frac{1}{p}\mathbb N$-graded, by defining the degree of $s^{1/p}$ to be $\frac{1}{p} \deg s$. Consider the graded $S$-submodule 
$[S^{1/p}]_{\mathbb N}$
of $S^{1/p}$ of elements of integer degree: this includes all the elements of $S$, but also elements of the form $(s)^{1/p},$ where $s$ is {\it not\/} a $p$-th power in $S$ but its degree is a multiple of $p$.  The graded map of $S$-modules 
$$
S \hookrightarrow  [S^{1/p}]_{\mathbb N} $$ 
 corresponds to the Frobenius 
 map of coherent sheaves
$\mathcal O_X \rightarrow F_*\mathcal O_X$ on $X$. The point now is that restricting the
 map $S_{\mathcal L}^{1/p} \overset{\pi}{\rightarrow} S_{\mathcal L}$  to the subgroup 
$[S^{1/p}]_{\mathbb N}$, the composition of maps of graded $S$-modules
$$
S \hookrightarrow  [S^{1/p}]_{\mathbb N} \overset{\pi}\rightarrow S
$$ gives a graded splitting of $S$-modules, 
whose  corresponding map of coherent sheaves on $X$ gives a splitting of Frobenius for $X$. The proof for global F-regularity is similar.
See \cite{SmithGloballyFReg} or \cite{SchwedeSmithLogFanoVsGloballyFRegular} for details. 
\end{proof}

\subsection{Frobenius splittings and anticanonical divisors.}\label{antican}
Summarizing the situation for curves,  we see that the existence of Frobenius splittings appears to be related to positivity of the anti-canonical divisor: 
\begin{ex}\label{curvesummary}
Among smooth projective curves over a perfect field of prime characteristic:
\begin{enumerate}
\item A genus zero curve  is {\it always\/}  globally F-regular.  
\item A genus one curve is never globally F-regular, and it is  Frobenius split if and only if it is an ordinary elliptic curve.
\item Higher genus curves are {\it never} Frobenius split (hence nor globally F-regular). \end{enumerate}
\end{ex}

 Indeed, there is a natural sense in which the sheaf of ``potential Frobenius splittings" is a sheaf of pluri-anticanonical forms. The following crucial fact, first appearing in this guise in \cite{Mehta-Ramanathan},  is at the heart of many ideas in both the local and global theories:

\begin{Lem}\label{Duality}
Let $X$ be a normal
{\footnote{If $X$ is not smooth,  the notation $\omega_X^n$ denotes the unique reflexive sheaf which agrees with the $n$-th tensor power of  $\omega_X$ on the smooth locus; equivalently, it is the double dual of the $n$-th tensor power of $\omega_X$, or equivalently, $\mathcal O_X(nK_X)$ where $K_X$ is the Weil divisor agreeing with a canonical divisor on the smooth locus.}}
 projective variety over a perfect field. Then we have
\[\mathcal{H}om_{\mathcal O_X}(F^e_*\mathcal{O}_X,\mathcal{O}_X)\cong F^e_*\omega_X^{1-p^e}.\] 
\end{Lem}

\begin{proof}
First assume that $X$ is smooth. 
We have  a natural isomorphism
$$ \mathcal{H}om_{\mathcal O_X}(F^e_*\mathcal{O}_X,\mathcal{O}_X)  \cong 
  \mathcal{H}om_{\mathcal O_X}(F^e_*\mathcal{O}_X,\omega_X) \otimes\omega_X^{-1}.$$ By Grothendieck duality for the finite map $F^e:X \rightarrow X$, we also have a natural isomorphism 
$$ \mathcal{H}om_{\mathcal O_X}(F^e_*\mathcal{O}_X,\omega_X) \cong  F^e_*\mathcal{H}om_{\mathcal O_X}(\mathcal{O}_X,\omega_X) \cong F^e_* \omega_X,$$
so that 
$$
\mathcal{H}om_{\mathcal O_X}(F^e_*\mathcal{O}_X,\mathcal{O}_X)\cong F^e_* \omega_X \otimes \omega_X^{-1} \cong F^e_*(\omega_X \otimes F^{e*}\omega_X^{-1}),$$
with the last isomorphism coming from the projection formula. 
Finally we have
$$\mathcal{H}om_{\mathcal O_X}(F^e_*\mathcal{O}_X,\mathcal{O}_X)\cong F^e_*(\omega_X^{1-p^e}),
$$ since pulling back an invertible sheaf under Frobenius amounts to raising it to the $p$-th power.

Now, even if $X$ is not smooth,  this proof is essentially valid. Indeed, we can carry out the same argument on the smooth locus of $X$, to produce the desired natural  isomorphism of sheaves there. Since both   sheaves $\mathcal{H}om_{\mathcal O_X}(F^e_*\mathcal{O}_X,\mathcal{O}_X)$ and $ F^e_*(\omega_X^{1-p^e})
$ are reflexive sheaves on the normal variety $X$,  this isomorphism extends uniquely to an isomorphism of ${\mathcal{O}_X}$-modules over all $X$. 
\end{proof}

Any Frobenius splitting   is a map $F_*\mathcal O_X \rightarrow \mathcal O_X,$ and hence a non-zero global section of $\mathcal{H}om_{\mathcal O_X}(F_*\mathcal{O}_X,\mathcal{O}_X) \cong F_*(\omega_X^{1-p}),
$ which in turn, is a non-zero section of $H^0(X, \omega_X^{1-p})$.  Thus, if $X$ is Frobenius split, we expect non-zero sections of $\omega_X^{1-p}$.

If $\omega_X$ is ample, then the sheaves $\omega_X^{1-p}$ are dual to ample, and can have no global sections: 

\begin{Cor}
A smooth projective variety $X$  with ample canonical bundle is never Frobenius split.
\end{Cor}

Even if $\omega_X^{1-p}$ has global sections,  it is not so obvious which of these might correspond to a splitting of Frobenius. 
 Mehta and Ramanathan  \cite[Prop 6]{Mehta-Ramanathan} found  a nice criterion: 
  
 \begin{Prop}\label{splitcrit} Let $X$ be a normal projective variety. 
 A section $s \in   H^0(X, \omega_X^{1-p})$ corresponds to a Frobenius splitting of $X$ if and only if there exists a smooth point $x \in X$ at which $s$ has a non-zero residue.  Explicitly, for a global  section $s \in H^0(X, \omega_X^{1-p})$, we can write the germ of $s$ at the point $x$ as  $s = f (dx_1\wedge \dots \wedge dx_n)^{1-p}$ where $x_1, \dots, x_n$ are a regular sequence of parameters at $x$ and $f \in \mathcal O_{X,x}$. 
 Now $s$ is a splitting of Frobenius if and only if 
  the  power series expansion of $f$ in the coordinates $x_i$ has a non-zero  $(x_1x_2\dots x_n)^{p-1}$ term. 
  \end{Prop}

Summarizing this in the language of divisors:  the non-zero mappings $F_*\mathcal O_X \rightarrow \mathcal O_X$ correspond to effective divisors in the linear system $|(1-p)K_X|. $ 
Given a particular divisor $D$ in $|(1-p)K_X|$, it is a splitting of Frobenius if and only if in local analytic coordinates at some smooth point $x \in X$, the divisor $D$ is $(p-1)$ times a simple normal crossing divisor whose components intersect exactly in $\{x\}$. 

\begin{ex}\label{toric} One easy case  in which  Frobenius splitting can be established using Proposition \ref{splitcrit} 
is when   a smooth projective variety $X$ of dimension $n$  admits $n$ effective divisors $D_1, \dots, D_n$ meeting transversely at a point of $X$ and  whose sum is an  anti canonical divisor. Projective space  obviously has this property (taking $-K_X$ to be the sum of the coordinate hyperplanes).  Similarly, a projective toric variety is Frobenius split as well, since  $-K_X$ is the sum of all the torus invariant divisors \cite[page 85]{FultonToricBook}.  Again, we recover that fact that smooth projective toric varieties are Frobenius split. Cf. Example \ref{toric}.
\end{ex}

Similarly,  if $X$ is globally F-regular, we expect many global sections of $\omega_X^{1-p^e}$:  for each $D$, the splitting $  F^e_* \mathcal O_X(D) \overset{t}\rightarrow \mathcal O_X$ induces a different splitting $F^e_*\mathcal O_X \rightarrow \mathcal O_X$, so  gives rise to a non-zero element global section of $\omega_X^{1-p^e}$.   Varying $D$, we get many sections of $\omega_X^{1-p^e}$--- so many that they grow polynomially in $p^e$:

\begin{Cor}\label{bigantican}
If a smooth projective variety is globally F-regular, then its anti-canonical bundle
 is big.{\footnote{A line bundle $\mathcal L$ on a projective variety $X$ is {\it big\/} if the space of  global sections $H^0(X, \mathcal L^n)$ grows as a polynomial of degree $\dim X$ in $n$; See \cite[Section 2.2]{LazPAG1}.}}
\end{Cor}

\begin{proof} 
For  any Weil divisor $D$ on a normal variety $X$,  the proof of Lemma \ref{Duality} immediately generalizes to give a 
 natural isomorphism
$$\mathcal{H}om_{\mathcal O_X}(F^e_*\mathcal{O}_X(D),\mathcal{O}_X)  \cong F_*^e(\omega_X^{1-p^e}(-D)). $$
Now, let $X$ be globally F-regular, and let $A$ be any ample effective Cartier divisor. By definition, there exists an $e$ and a global non-zero map  $t:  F^e_*\mathcal O_X(A) \overset{t}{\rightarrow}   \mathcal O_X$ splitting 
  the  composition 
 $
 \mathcal O_X \rightarrow 
  F^e_*\mathcal O_X(A).$ The map $t$ can be viewed thus be viewed as a non-zero global section of 
  $ \mathcal{H}om_{\mathcal O_X}(F^e_*\mathcal{O}_X(A),\mathcal{O}_X)$, hence a non-zero element of 
  $$ H^0(X, F_*^e(\omega_X^{1-p^e}(-A))) = H^0(X, \omega_X^{1-p^e}(-A)) \subset  H^0(X, \omega_X^{1-p^e}). $$ Let $E$ be  the effective divisor  of the section $t$, so that 
  $ E \in |(1-p^e)K_X - A|$, whereby  
  $$ E + A \in  |(1-p^e)K_X|.$$
  Finally, we see that $-K_X$ is $\mathbb Q$-linearly equivalent to $\frac{1}{p^e-1}(A + E)$, so that $-K_X$ is  $\mathbb Q$-linearly equivalent to ``ample plus effective." That is, $-K_X$ is big by Corollary 2.2.7 in \cite{LazPAG1}.
  \end{proof}

Unfortunately, bigness  of $ \omega_X^{1-p^e}$ is not sufficient for global  F-regularity.
For example,  a ruled surface over an elliptic curve is never strongly F-regular, but its 
anti-canonical divisor can be big. 
See \cite[Example 6.7]{SchwedeSmithLogFanoVsGloballyFRegular}.

On the other hand,  there is strengthened form of ``almost amplitude  of $-K_X$" which guarantees enough good sections of $\omega_X^{1-p^e}$ to find lots of splittings of Frobenius:

\begin{defi}\label{logFanodef}
A normal projective variety $X$ is  {\it log Fano} if there exists an effective $\mathbb Q$-divisor $\Delta$ on $X$ such that 
\begin{enumerate}
\item $-K_X-\Delta$ is ample; and
\item the pair $(X,\Delta)$ has (at worst)  Kawamata log terminal singularities\footnote{Kawamata log terminal singularity is usually defined in characteristic 0, but it can be defined in any characteristic
by considering {\it all} birational proper maps as follows. Let $X$ be a normal variety and $\Delta$ be an effective $\Q$-divisor on $X$. Then $(X,\Delta)$ is called {\it Kawamata log terminal} if 
\begin{enumerate}
\item $K_X+\Delta$ is $\Q$-Cartier; and
\item for {\it all} birational proper maps $\pi:Y\to X$, choosing $K_Y$ so that $\pi_*K_Y = K_X$, each coefficient of $\pi^*(K_X+\Delta)-K_Y$ is strictly less than 1.
\end{enumerate} 
}.
\end{enumerate} 
\end{defi}

If $X$ is smooth and  $\omega_X^{-1}$ ample (that is, if $X$ is Fano), then $X$ is log Fano: 
we can take $\Delta =0$.  
For log Fano varieties in general, $-K_X$ is not ample,  but it is close to ample in the sense that it is big and even more:  it is ``close" to the ample cone in the sense that we can obtain it  from  the ample divisor ($-K_X - \Delta$) by adding only the very small effective $\mathbb Q$-divisor $\Delta$ whose singularities are highly controlled.

We can now state a pair of theorems which can be viewed as a sort of  geometric characterization of globally F-regular varieties:

\begin{Thm}\label{FanoGlob}
 \cite{SmithGloballyFReg} \cite{SchwedeSmithLogFanoVsGloballyFRegular}
If $X$ is a globally F-regular projective variety of characteristic $p$, then $X$ is log Fano. \end{Thm}

The converse isn't quite true because of irregularities in small characteristic. For example, the cubic hypersurface defined by $x^3 + y^3 + z^3 + z^3$ in $\mathbb P^3$ is obviously a smooth Fano variety (hence log Fano) in every characteristic $p \neq 3$. But it is not globally F-regular or even Frobenius split in characteristic two. However, it {\it is\/} globally F-regular for all characteristics $p \geq 5$.  
In general we have

\begin{Thm}\label{FanoGlobchar0} \cite{SmithGloballyFReg} \cite{SchwedeSmithLogFanoVsGloballyFRegular}
If $X$ is a log Fano variety of characteristic zero, then $X$ has globally F-regular type.
\end{Thm}

Remarkably, the converse to Theorem \ref{FanoGlobchar0} is open: we do not know whether a globally F-regular type variety must be log Fano. This may seem surprising at first glance.
If $X$ has globally F-regular type, then in each characteristic $p$ model, the proof of Theorem \ref{FanoGlob} constructs a  ``witness" divisor $\Delta_p$ establishing that the pair  $(X_p, \Delta_p) $ is log Fano. But $\Delta_p$ depends on $p$ and there is no {\it a priori} reason that the $\Delta_p$ all come from some divisor $\Delta$ on the characteristic zero variety $X$.

\begin{Conj}\label{logFanoconj}
A projective globally F-regular type variety (of characteristic zero) is log Fano. 
\end{Conj}

Gongyo, Okawa, Sannai and Takagi prove Conjecture \ref{logFanoconj} under the additional hypothesis that the variety is a $\mathbb Q$-factorial  Mori Dream space, by applying the minimal model program \cite{GongyoOkawaSannaiTakagi}. This gives urgency to another interesting question:  are 
globally F-regular type varieties (of characteristic zero)  Mori Dream Spaces? Moreover, since  log Fano spaces (of characteristic zero) are Mori Dream spaces by \cite[Cor 1.3.2]{BCHM10}, the answer is necessarily {\it yes} if Conjecture \ref{logFanoconj} is true. 
What about in characteristic $p$? 
\begin{Question}
Assume that $X$ is globally F-regular. Is it true that the Picard group of $X$ is  finitely generated? Is it true that the Cox ring of $X$ is always finitely generated?
\end{Question}

Similarly, there are related questions and results about the geometry of Frobenius split projective varieties. For example,

\begin{Thm}\label{logCY}\cite{SchwedeSmithLogFanoVsGloballyFRegular}
If $X$ is a normal Frobenius split  projective variety of characteristic $p$, then $X$ is log Calabi-Yau. 
This means that $X$ admits an effective $\mathbb Q$-divisor such that $(X, \Delta)$ is log canonical
\footnote{Log canonical  is usually defined in characteristic 0, but it can be defined in any characteristic similarly to how we defined Kawamata log terminal singularities. We require instead that   each coefficient of $\pi^*(K_X+\Delta)-K_Y$ is at most 1.
}
 and 
$K_X + \Delta$ is $\mathbb Q$-linearly equivalent to the trivial divisor. 
\end{Thm}

Again,  the converse fails because of irregularities in small characteristic. The same cubic hypersurface defined by $x^3 + y^3 + z^3 + z^3$ in $\mathbb P^3$ is not Frobenius split in characteristic two, but it is  a smooth 
log Calabi Yau variety.  However, we do expect an analog of Theorem \ref{FanoGlobchar0} to hold. 
We conjecture:
\begin{Conj}\label{logCYchar0}  \cite{SchwedeSmithLogFanoVsGloballyFRegular}.
If $X$ is a log Calabi Yau variety of characteristic zero, then $X$ has Frobenius split type.
\end{Conj}

\subsection{Pairs.}

We have discussed Frobenius splitting and F-regularity in an absolute setting: these were defined as properties of a scheme $X$. However, in the  decade since the last MSRI special year in commutative algebra, a theory of ``F-singularities of pairs" has  flourished, inspired by the rich theory of pairs developed in birational geometry \cite{KollarSantaCruz}.   The idea to
extend Frobenius splitting and F-regularity to  pairs was a major breakthrough, pioneered by Nobuo Hara and Kei-ichi Watanabe in \cite{HaraWatanabe}.   Although we do not have space to include a careful treatment of this generalization here, we briefly outline the definitions and main ideas. 

By {\it pair\/} in this context, we have in mind a normal irreducible scheme $X$ of finite type over a perfect field,  together with  a  $\mathbb Q$-divisor $\Delta$ on $X$. (Another variant considers pairs $(X, \mathfrak a^t)$ consisting of an ambient scheme $X$ together with a sheaf of ideals $\mathfrak a$ and a rational exponent $t$.{\footnote{There are even triples $(X, \Delta, \mathfrak a^t)$ incorporating aspects of both variants.}})  In the geometric setting, an additional assumption---namely that $K_X + \Delta$ is $\mathbb Q$-Cartier---is usually imposed, because one of the main techniques in birational geometry involves pulling back divisors to different birational models (and only $\mathbb Q$-Cartier divisors can be sensible pulled back). 
One possible advantage of the algebraic notion of pairs is that it is not necessary to assume that $K_X+\Delta$ is $\mathbb Q$-Cartier, although alternatives have also been proposed directly in the world of birational geometry as well; see \cite{deFernexHaconSingNormalVar}. 

\subsubsection{
What do pairs have to do with Frobenius splittings?}  Given an $\mathcal O_X$-linear map $\phi: F^e_*\mathcal O_X \rightarrow \mathcal O_X$, we have already seen how it gives rise to a global section of $F^e_*\omega_X^{1-p^e}, $ hence an effective divisor $\tilde D$ linearly equivalent to $(1-p^e)K_X$. If we set $\Delta = \frac{1}{p^e-1}\tilde D$, the pair $(X, \Delta)$ can be interpreted as more or less equivalent to the data of the map $\phi$.  If $D$ is an effective Cartier divisor through which our 
map $\phi$ factors, then (as in the proof of Corollary \ref{bigantican}), we can view $\phi$ as a global section of $\mathcal O_X((1-p^e)K_X + D)$, hence an effective divisor $\tilde D$ linearly equivalent to $(1-p^e)K_X + D$. Setting $\Delta = \frac{1}{p^e-1}\tilde D$, we have that $\Delta$ is a $\mathbb Q$-divisor satisfying $K_X + \Delta$ is $\mathbb Q$-Cartier, since it is $\mathbb Q$-linearly equivalent to the Cartier $\mathbb Q$-divisor 
$\frac{1}{p^e-1}D$.
We refer to the exceptionally clear exposition of this idea, with deep applications to understanding the behavior of test ideals under finite morphisms, in the paper \cite{SchwedeTuckerTestidealsFiniteMorphism}.

The definition of  (local or global) F-regularity and Frobenius splitting can be generalized to pairs as follows:

\begin{defi} Let $X$ be a normal F-finite variety, and $\Delta$ an effective  $\mathbb Q$-divisor on $X$.
\begin{enumerate} 
\item 
The pair $(X, \Delta)$ is sharply Frobenius split  (respectively locally sharply Frobenius split) if  there exists an $e \in \mathbb N$ such that the natural map 
$$
\mathcal O_X \rightarrow F_*^{p^e}\mathcal O_X(\lceil{(p^e-1) \Delta}\rceil)
$$ splits as an map of sheaves of $\mathcal O_X$-modules (respectively, splits locally at each stalk).
\item The pair $(X, \Delta)$ is globally (respectively, locally) F-regular if for all effective divisors $D$,  there exists an $e \in \mathbb N$ such that the natural map 
$$
\mathcal O_X \rightarrow F_*^{p^e}\mathcal O_X(\lceil{(p^e-1) \Delta}\rceil + D)
$$ splits as an map of sheaves of $\mathcal O_X$-modules (respectively, splits locally at each stalk).

\end{enumerate}
\end{defi}

\begin{Rmk} A slightly different definition of Frobenius splitting for a pair $(X, \Delta)$ was first given by Hara and Watanabe \cite{HaraWatanabe}. The variant here, which fits better into our context,   was introduced by Karl Schwede \cite{SchwedeRefinementsOfFRegularity}.  
\end{Rmk}

The local properties of F-regularity and F-purity for pairs  turn out to be closely related to the properties of log terminality and log canonicity that arose independently, in the minimal model program,  in the 1980's.  The absolute versions of the following theorems were already mentioned at the end of the first Lecture.

\begin{Thm}\cite{HaraWatanabe}\label{equivofsings}
Let $(X, \Delta)$ be a pair where $X$ is a normal variety of prime characteristic and $\Delta$ is a $\mathbb Q$-divisor such that $K_X + \Delta$ is $\mathbb Q$-Cartier. 
\begin{enumerate} \item If $(X, \Delta)$ is a locally F-regular pair, then $(X, \Delta)$ is Kawamata Log terminal.
\item If $(X, \Delta)$ is a locally  sharply Frobenius split  pair, then $(X, \Delta)$ is log canonical.
\end{enumerate}
\end{Thm}

Similarly, there are global versions:     Theorem \ref{FanoGlob} and \ref{logCY} also hold for ``pairs." 
See  \cite{SchwedeSmithLogFanoVsGloballyFRegular}.

 In characteristic zero,  the discussion of Section \ref{char0stuff} generalizes easily to  pairs.  
 Given a pair $(X, \Delta)$, where now $X$ is normal and essentially finite type over a field of characteristic {\it zero}, we can define the {\it pair\/} to be (locally or globally) {\bf F-regular type } or 
(locally or globally) {\bf sharply Frobenius split type}   as in Definition \ref{cha0}. 
That is, we chose a  ground ring $A$ over which both $X$ and $\Delta$ are defined, which gives rise to a pair $(X_A, \Delta_A)$ over $X_A$, and define  the pair $(X, \Delta)$ to be of 
(locally or globally) F-regular-type if for a Zariski dense set of closed points in Spec $A$, the 
pair $(X_A \times Spec A/\mu, \Delta_A {\text{ mod } }\mu ) $ is (locally or globally) F-regular, where $  \Delta_A {\text{ mod } } \mu$ denotes the pullback of $\Delta_A$ to the closed fiber $X_A \times Spec A/\mu$. Similarly, we define sharp Frobenius splitting type.
For details of this reduction to prime characteristic, see for example, \cite{HaraWatanabe}.

With the definitions in place, Theorem \ref{equivofsings}  implies  characteristic zero versions. Let 
 $(X, \Delta)$  be a pair where $X$ is a normal variety of prime characteristic zero and $\Delta$ is a $\mathbb Q$-divisor satisfying   $K_X + \Delta$ is $\mathbb Q$-Cartier. 
Then if  $(X, \Delta)$ is of locally F-regular type, then $(X, \Delta)$ has klt singularities \cite{HaraWatanabe}.  See the cited papers for details. 

As in the absolute case, there are  converses (some conjectured) to all these results for pairs in characteristic zero. Hara and Watanabe show that klt pair $(X, \Delta)$ of characteristic zero has F-regular type \cite{HaraWatanabe}; this leans heavily on some injectivity results for Frobenius acting on cohomology groups proved in  \cite{HaraCharacterizationRatSing}; see also \cite{MehtaSrinivasRatImpliesFRat}.  In  \cite{SchwedeSmithLogFanoVsGloballyFRegular}, these results used to prove a global version: a log Fano pair $(X, \Delta)$ of characteristic zero is of globally F-regular type.  The log canonical property has proved more elusive: it is conjectured that if a  pair $(X,\Delta)$ of characteristic zero is log canonical, then it is of locally sharply Frobenius split type, but this questions has  remained open since its inception in the nineties. See \cite{HaraWatanabe}.
If this local conjecture holds, then the global analog follows: a Calabi-Yau pair $(X, \Delta)$ of characteristic zero is of globally sharply Frobenius split type; see  \cite{SchwedeSmithLogFanoVsGloballyFRegular}.

We can think of  F-regularity as a ``characteristic $p$ analog" of klt singularities, and (at least conjecturally) F-splitting as a ``characteristic $p$ analog"  of log canonical singularities. The analogy runs deep: F-pure thresholds become  ``characteristic $p$ analogs" of log canonical thresholds, test ideals become  ``characteristic $p$ analogs" of multiplier ideals, centers of sharp F-purity become  ``characteristic $p$ analogs"  of  log canonicity, F-injectivity becomes a ``characteristic $p$ analog" of  Du Bois singularities.

\section{The Test Ideal.}

 The {\it test ideal\/} is a distinguished ideal  reflecting  the Frobenius properties of  a prime characteristic ring. 
For example,  the test ideal
defines the closed locus of  Spec $R$ consisting of points $\mathfrak p$ at which $R_{\mathfrak p}$ is not F-regular.  Test ideals are ``characteristic $p$ analogs" of multiplier ideals in birational algebraic geometry (\cite{SmithMulplierIdealUniversalTestIdeal} and \cite{HaraGeometricInterpTestideals});  they define a distinguished ``compatibly split subscheme"  of a Frobenius split variety (\cite{VassilevTestIdeals} and \cite{SchwedeCentersOfFPurity}).  

Test ideals can be defined very generally for pairs on more or less arbitrary Noetherian schemes of characteristic $p$.  However, the  theory becomes most  transparent in two special cases, which are loosely the ``classical commutative algebra case" and  the ``classical algebraic geometry case."
In the classical commutative algebra case,  the scheme is  the spectrum of a local ring $R$ and 
we are interested in the ``absolute" test ideal.  In this case, the test ideal $\tau(R)$ is essentially Hochster and Huneke's test ideal for tight closure.{\footnote{Our terminology differs slightly from the tight closure literature, where our test ideal would be called the  ``big test ideal" for ``non-finitistic tight closure;" also, it defines the non-strongly F-regular locus.}} 
In the classical algebraic geometry case, we are interested in the test ideal of a pair $(X, \Delta)$, where $X$ is a smooth ambient scheme  and $\Delta$ is an effective $\mathbb Q$-divisor on $X$ 
(or $\Delta =  \mathfrak a^t$ where $\mathfrak a$ is an ideal sheaf on $X$ and $t$ is a positive rational number).  In this case, the  test ideal $\tau(X, \Delta)$ turns out to be a  ``characteristic $p$ analog" of the multiplier ideal in complex algebraic geometry, a subject  described, for example,  in Rob Lazarsfeld's MSRI introductory talks ten years ago \cite{BlickleLazarsfeldMultiplier}.

In this lecture, we explain the test ideal in  the ``classical commutative algebra" setting. We develop the test ideal  as   just one ideal   in a  lattice of  ideals   distinguished  with respect to the Frobenius map.  Our definition is not the traditional one  due to Hochster and Huneke, but 
a newer twist (due essentially to Schwede \cite{SchwedeCentersOfFPurity}) which is both illuminating and elegant,  tying the ideas into Mehta and Ramanathan's theory of compatibly split ideals.  The fourth lecture will treat test ideals in the ``classical algebro-geometric" setting.

\subsection{Compatible Ideals.} 
Let $R$ be an F-finite reduced  ring of characteristic $p$.
\begin{defi}\label{compatible}  Fix any 
 $R$-linear map $\varphi:R^{1/p^e}\to R$. 
An ideal $J$ of $R$ is called {\it $\varphi$-compatible} if $\varphi(J^{1/p^e})\subseteq J$.
\end{defi}

Put differently,  given an $R$-linear map $\varphi:  R^{1/p^e} \rightarrow R,$  consider the  obvious diagram
$$
\xymatrix{
R^{1/p^e}   \ar@{->}[r]^{{\varphi}}    \ar@{>>}[d]& R   \ar@{>>}[d]\\
 (R/J)^{1/p^e}  \ar@{.>}[r]   &   R/J, 
}
$$
where the vertical arrows are the natural surjections. 
The bottom arrow can not be filled in to make a commutative diagram in general: it  can be filled in  if and only if $J$ is $\varphi$-compatible.  That is, an ideal  $J$ is $\varphi$-compatible if and only if the map $\varphi: R^{1/p^e} \rightarrow R$ descends to a map $(R/J)^{1/p^e}\rightarrow  R/J$.   

\begin{ex}
\label{exercise: compatible ideals two variables}
Let $R=\mathbb{F}_p[x,y]$ and let $\phi:R^{1/p}\to R$ be the $R$-linear splitting defined by
\[\phi(x^{\frac{i}{p}}y^{\frac{j}{p}})=\begin{cases}x^{\frac{i}{p}}y^{\frac{j}{p}} & \frac{i}{p},\frac{j}{p}\in\mathbb{N}\\0 & {\rm otherwise.}\end{cases}\]
As an exercise, the reader should check that  the ideals $(0), (x), (y),  (xy), $ 
and $(x, y)$
are all $\phi$-compatible---in fact, they are the only $\phi$-compatible ideals in $R$. \end{ex}

The following properties are straightforward: 

\begin{Prop} \label{compprop} Fix an $R$-linear map $\varphi: R^{1/p^e} \rightarrow R$. 
\begin{enumerate}
\item
The set of $\varphi$-compatible  ideals is closed under sum and intersection. 
\item The minimal primes of a $\varphi$-compatible ideal are $\varphi$-compatible.
\item If $\varphi$ is a Frobenius splitting, then all $\varphi$-compatible ideals are radical. 
\end{enumerate}
\end{Prop}

\begin{proof}
We leave (1) as an easy exercise. For statement (2), let $P$ be a minimal prime of a $\varphi$-compatible ideal $J$. Take any  $w$ not in $P$ but  in the intersection of all other primary components of $J$, that is,  take $w  \in (J:P) \setminus P$. 
For any $z \in P$, we need to show that $\varphi(z^{1/p^e}) \in P$. 
Now since
$
w^{p^e}z \in J 
$ and $J$ is $\phi$-compatible, we have
$$
 w \varphi(z^{1/p^e}) = \varphi(wz^{1/p^e})  \in \varphi(J^{1/p^e}) \subset J \subset P.$$
Since $w \notin P$, we conclude that $\varphi(z^{1/p^e}) \in P$.
So $P$ is $\varphi$-compatible. 

Statement (3) is also easy: if $J$ is $\phi$ compatible, 
the  commutative diagram
$$
\xymatrix{
R^{1/p^e}   \ar@{->}[r]^{{\varphi}}    \ar@{>>}[d]& R   \ar@{>>}[d]\\
 (R/J)^{1/p^e}  \ar[r]  &   R/J, 
}
$$
shows that if $\varphi$ is a Frobenius splitting, so is the induced map on $R/J$. Since Frobenius split rings are reduced, the  ideal $J$ is radical. 
\end{proof}

\subsubsection{Compatibly Split Subschemes.}
When $\varphi$ is a Frobenius splitting,   a $\varphi$-compatible ideal is often called a $\varphi$-compatibly split ideal, or (suppressing the dependence on $\varphi$) a compatibly split ideal. The subscheme it defines is called {\it  a compatibly split subscheme.} 
  The notion of {\it compatible Frobenius splitting\/} was first introduced in \cite{Mehta-Ramanathan}. In that language, a compatibly split ideal (sheaf) on a Frobenius split variety defines a compatibly split subscheme---a subscheme to which the given splitting of Frobenius restricts.

 For a fixed splitting $\varphi$ of Frobenius, the set of $\varphi$-compatibly split ideals is always a  {\it finite\/} set of radical ideals. The finiteness is not obvious;  see \cite{SchwedeFAdjunction} or \cite{KumarMehtaFinitenessofthenumberofcompatiblysplitsubvarieties}, or \cite{SchwedeTuckerNumberOfCompatibleSubvarieties}. Also \cite{EnescuHochsterTheFrobeniusStructureOfLocalCohomology} and \cite{SharpGradedAnnihilators} contains a related dual fact.

\subsection{Uniformly Compatible Ideals.}
 Of course, if $R$ is Frobenius split, there may be many different splittings of Frobenius.  Different splittings   produce different compatibly split ideals. For example,  by linearly changing coordinates in $\mathbb F_p[x, y]$, we can construct a different splitting of Frobenius, call it $\varphi_{ab}$,  much like the one in Example \ref{exercise: compatible ideals two variables} but centered instead on the point $(x-a, y-b)$. Its compatibly  split ideals will be the ideals  $(x-a, y-b)$, $(x-a)(x-b)$,  $(x-a), (y-b), $ and the zero ideal. These ideals are compatibly split with respect to $\varphi_{ab}$ but not with the $\varphi$ from Example 
  \ref{exercise: compatible ideals two variables}.

The ideals which are compatible with respect to  {\it every\/}  $R$-linear map $R^{1/p^e} \rightarrow R$ play an essential role in our story:
\begin{defi}An ideal $J$ in an F-finite ring is {\bf uniformly $F$ compatible} if it is compatible with respect to {\it every\/}  $R$-linear map $R^{1/p^e} \rightarrow R$, for all $e$.  
\end{defi}

  The test ideal is  a  distinguished uniformly $F$ compatible ideal:

\begin{defi}
\label{defn: test ideal classical}
The test ideal{\footnote{Again a reminder for experts in tight closure: this is equal to the  ``big" test ideal in the tight closure terminology.  If $R$ is complete local, for example, the test ideal we define here is the same as the annihilator of the non-finitistic tight closure of zero in the injective hull of the residue field of $R$ \cite{LyubeznikSmith2}. Of course, all versions of test ideals in the tight closure theory are conjectured to be equal, and are known to be equal in many cases, including for Gorenstein $R$ and graded $R$ \cite{LyubeznikSmithStrongWeakFregularityEquivalentforGraded}.}} of an F-finite Noetherian domain $R$
 is the {\it smallest \/} non-zero uniformly $F$ compatible  ideal. 
That is, the test ideal is the smallest non-zero ideal $J$ 
  that satisfies
\[\varphi(J^{1/p^e})\subseteq J\]
for all $\varphi\in \Hom_R(R^{1/p^e},R)$ and all $e\geq 1$.
More generally, if $R$ is not a domain, we define the test ideal as the smallest uniformly $F$ compatible ideal not contained in any minimal prime. 
\end{defi}

{\it It is important to note  in Definition \ref{defn: test ideal classical} that it is {\bf not at all obvious that there  exists\/} a smallest such ideal:}
why couldn't the set of all compatible ideals include an infinite descending chain of non-zero ideals?
This is a deep result, essentially relying on an important lemma of Hochster and Huneke crucial to their proof of the  existence of ``completely stable test elements" (Cf. the proof of  Theorem \ref{testelement}).  
For a  summary of the proof, see the survey \cite{SchwedeTuckerSurvey}.

\begin{Rmk}
The set of uniformly $F$ compatibly ideals forms a lattice closed under sum and intersection, according to Proposition \ref{compprop}. 
 This  lattice has been studied before: in the local Gorenstein case, it is the precisely the lattice of F-ideals discussed in   \cite{MR1248078}; more generally,  it is the  lattice of annihilators of $\mathcal F(E)$-modules in \cite{LyubeznikSmith2}. However, those sources define the ideals as annihilators of certain  Artinian $R$-modules with Frobenius action.  Schwede's insight was that these ideals could be defined directly (dually to the original emphasis), thereby  producing   a more straightforward and global  theory which  neatly ties in with   Mehta and Ramanathan's  ideas  on compatible Frobenius splitting.
\end{Rmk}

\begin{Rmk}\label{Fideals} Experts in tight closure can  see easily how the  definition of the test ideal here relates to the one in the literature, and why there is a unique smallest uniformly $F$ compatible ideal, at least in the Gorenstein local case.
Let $(R, m)$ be a Gorenstein local domain of dimension $d$. As is well-known, the test ideal is the  annihilator of the tight closure of zero in $H^d_m(R)$. In \cite{SmithRationalSingularities},  the Frobenius stable submodules of  $H^d_m(R)$ (including the tight closure of zero) are analyzed and their annihilators in $R$ are dubbed ``F-ideals;" there it is shown (also using test elements!)  that there is a unique largest proper Frobenius stable submodule of $H^d_m(R)$, hence a unique smallest  non-zero F-ideal, namely test ideal of $R$.  The uniformly $F$ compatible ideals are precisely the F-ideals---that is, annihilators of submodules of the top local cohomology module $H^d_m(R)$ stable under Frobenius. 
This is not hard to check using Lemma \ref{generator}; see \cite{SchwedeCentersOfFPurity} or  \cite[Thm 4.1]{EnescuHochsterTheFrobeniusStructureOfLocalCohomology}.
 The non-Gorenstein case is treated  in \cite{LyubeznikSmith2}; the uniformly $F$ compatible ideals are the annihilators of the $\mathcal F(E)$-modules there.  Schwede includes a fairly comprehensive discussion of the connections between his uniformly  $F$ compatible ideals and existing ideas in the literature; see 
   \cite{SchwedeCentersOfFPurity}.  
   \end{Rmk}


\begin{ex} The test ideal of $\mathbb F_p[x, y]$ is the whole ring. Indeed, we have seen that  every non-zero $c \in \mathbb F_p[x, y]$ can be taken to $1$ by some $R^{p^e}$-linear map.  So no non-zero proper ideal is uniformly F-compatible.
\end{ex}

\begin{Thm}\label{localization} Let $R$ be a reduced F-finite of characteristic $p>0$. 
\enumerate
\item The test ideal behaves well under localization and completion: for any multiplicative set $U$, the ideals $\tau(RU^{-1})$  and $\tau(R)U^{-1}$ coincide  in $RU^{-1}$, and for any prime ideal $\mathfrak p$, $\tau(\hat{R_{\mathfrak p}})\, = \,\tau(R)\hat{R_{\mathfrak p}}. $ 
\item  $R$ is  F-regular if and only if its test ideal is trivial.
\item
The test ideal defines the closed locus of prime ideals $\mathfrak p$ in Spec $R$ such that $R_{\mathfrak p}$ fails to be F-regular.
\end{Thm}

The proof uses the following important lemma.
\begin{lemma}\label{genlem}
Let $c$ be an element of  a reduced F-finite ring $R$. The   ideal  generated by all elements $\phi(c^{1/p^e})$ as we range over all $e$ and all $\phi \in  \Hom_R(R^{1/p^e},R)$ is uniformly F compatible. In particular, fixing any  $c \in \tau(R)$ but not in any minimal prime of $R$, the elements  $\phi(c^{1/p^e})$ generate $\tau$.
\end{lemma}
\begin{proof}[Proof of Lemma]
Let $J$ be the ideal generated by the $\phi(c^{1/p^e})$. 
We need to show that elements of the form $[r \phi(c^{1/p^e})]^{1/p^f}$ are taken into $J$ by any $\psi \in 
 \Hom_R(R^{1/p^f},R)$.  But 
 $$
 \psi[r^{1/p^f}[\phi(c^{1/p^e})]^{1/p^f}] =
 \psi[r^{1/p^f}[\phi^{1/p^f}(c^{1/p^{e+f}})] =
 ( \psi  \circ (\phi\circ r)^{1/p^f})(c^{1/p^{e+f}})
 $$ where $\psi  \circ (\phi \circ r)^{1/p^f}$ is the $R$-linear map 
 $$
 R^{1/p^{e+f} }\overset{r^{1/p^f}}\rightarrow  R^{1/p^{e+f} }\overset{\phi^{1/p^f}}\rightarrow 
R^{1/p^f} \overset{\psi}\rightarrow R.
$$
The second equality above is satisfied because $\phi^{1/p^f}$ is $R^{1/p^f}$ linear. The second statement of the Lemma follows by the minimality of the test ideal. The lemma is proved. 
\end{proof}

\begin{proof}[Proof of Theorem]
To prove (1),  the point is that $R^{1/p^e}$ is a finitely generated $R$-module, so  that  $$\Hom_{RU^{-1}}((RU^{-1})^{1/p^e},RU^{-1}) \cong 
 \Hom_R(R^{1/p^e},R) \otimes_R RU^{-1}.$$   
 Take any $c \in R$ (not in any minimal prime)  such that $c \in \tau(R)$ and $\frac{c}{1} \in \tau(RU^{-1}).$ By the Lemma above, 
   both $\tau(RU^{-1})$ and $\tau(R)U^{-1}$ are ideals of $RU^{-1}$ generated by elements of the form $\frac{\phi(c^{1/p^e})}{1} = \frac{\phi}{1}((\frac{c}{1})^{1/p^e})$ as we range through all 
  $\phi \in  \Hom_R(R^{1/p^e},R) $ and all  $e \in \mathbb N.$  That is, the ideals 
   $
 \tau(RU^{-1}))$ and $ \tau(R) U^{-1}
 $  coincide. 
  The second statement  follows similarly, since 
  $\Hom_{\hat R_\mathfrak p}(\hat R_{\mathfrak p}^{1/p^e}, \hat R_{\mathfrak p}) \cong 
 \Hom_{R_{\mathfrak p }}({R_{\mathfrak p}}^{1/p^e},R_{\mathfrak p}) \otimes_{R_{\mathfrak p} }\hat R_{\mathfrak p}.$

 For (2), assume for simplicity that  $R$ is a domain.{\footnote{Else replace ``non-zero" by ``not in any minimal prime" throughout.}} 
If $R$ is F-regular, then for any non-zero $c$, 
there exists $e$ and $\phi \in  \Hom_R(R^{1/p^e},R)$ such that $\phi(c^{1/p^e}) = 1$. 
This means that every uniformly F-compatible ideal contains $1$. In particular, $\tau(R)$ is trivial.  
Conversely,  assume $ \tau(R)$ is trivial. By (1), also  $\tau(R_m)$ is trivial  for each maximal ideal $m$. 
  For any non-zero element $c \in R_m$, the lemma implies that  
the elements $\phi(c^{1/p^e})$ can not be all contained in $m$ as  $\phi $ ranges over all $ \Hom_{R_m}(R_m^{1/p^e},R_m)$.  Thus there exists $\phi \in \Hom_{R_m}(R_m^{1/p^e},R_m)$ such that $\phi(c^{1/p^e})$ is a unit, and hence $R_m$ is F-regular. Since this holds for each maximal ideal, we conclude that $R$ is F-regular by Proposition \ref{Fregprop}. 

Statement (3) follows from (1) and (2) together. 
\end{proof}

The lattice of uniformly $F$ compatible ideals is especially nice in a Frobenius split ring. The following follows immediately from Proposition \ref{compprop} (3). 

\begin{Cor}
Every uniformly $F$ compatible ideal in a Frobenius split ring is radical. In particular, the test ideal in a Frobenius split ring is radical. 
\end{Cor}

The converse is not true: the ring $R=\frac{\mathbb F_p[x,y,z]}{(x^3 + y^3 + z^3)}$ has test ideal $(x, y, z)$ for all characteristics $p \neq 3$, but $R$ is not Frobenius split if $p = 2 \,{\text{mod}} \,3$. See \cite[Example 6.3]{SmithTestIdealsLocalRings}.

\subsection{Splitting Primes and Centers of F-purity.}
One might also wonder whether the {\it largest\/} proper  ideal compatible with respect to all $\phi$ might also be of interest? We know  a largest such exists by the Noetherian property of $R$, since the sum of uniformly $F$ compatible ideals is uniformly $F$ compatible. For Frobenius split $R$, 
this largest compatible ideal turns out to be the {\bf splitting prime\/} of Aberbach and Enescu  \cite{AberbachEnescuStructureOfFPure}. It is also the minimal center of F-purity in the language of Schwede {\cite{SchwedeCentersOfFPurity}}. The prime uniformly compatible ideals are what Schwede  calls {\it centers of F-purity}.  He shows that they are ``characteristic $p$ analogs" of Kawamata's centers of log canonicity, and that they satisfy an analog  Kawamata's subadjunction \cite{KawamataSubAdjII}. See \cite{SchwedeFAdjunction}.

\subsection{The Frobenius filtration of a Frobenius split ring.} 
We have already observed that when $R$ is a Frobenius split ring, the set of uniformly $F$ compatible ideals forms a (finite) lattice of radical ideals closed under addition and intersection. An interesting observation of 
Janet  Vassilev 
\cite{VassilevTestIdeals} creates a distinguished chain in this lattice.  

\begin{lemma} If $\tau$ is a uniformly $F$ compatible  ideal of $R$, then the pre-image in $R$ of any compatibly split ideal of $R/\tau$ is compatibly split in $R$.
\end{lemma}
\begin{proof}
Let $J$ be the preimage in $R$ of a uniformly compatibly split ideal of $R/\tau$. 
Let $R^{1/p^e} \overset{\phi}\rightarrow R$ be any $R$-module homomorphism. Because $\tau$ is uniformly $F$ compatible in $R$, there is an induced map of $R/\tau$-modules  $ (R/\tau)^{1/p^e} \overset{\bar\phi}\rightarrow R/\tau,$ and because $J/\tau$ is uniformly $F$ compatible in  $R/\tau,$ $\bar\phi$ descends to a map   $(R/J)^{1/p^e} \rightarrow R/J.$ But this is exactly what it means that $J$ is uniformly $F$ compatible in $R$.  
\end{proof}

To construct Vassilev's chain, start with 
 a Frobenius split ring $R$,  with  test ideal $\tau_0$. Because $\tau_0$ is compatible with respect to some (indeed, every) Frobenius splitting, the ring $R/\tau_0$ is also Frobenius split. Let $\tau_1$ be the preimage of the test ideal $\tau(R/\tau_0)$ of $R/\tau_0$ in $R$. By the Lemma, $\tau_1$ is also uniformly $F$ compatible, so $R/\tau_1$ is Frobenius split, and so its test ideal  lifts to an ideal $\tau_2$. Continuing in this way, we produce a chain $\tau_0 \subset \tau_1 \subset \dots \subset \tau_t$ of radical ideals, all uniformly $F$ compatible. 
Since the test ideal is never contained in a minimal prime, each ideal in the chain has strictly larger height than its predecessor, and since $\tau_0$ defines the non-F-regular locus of $R$, we see that the length of Vassilev's chain is bounded by the dimension of the non-F-regular set of $R$.


\subsection{Trace of Frobenius}\label{trace}
To check that an ideal is uniformly $F$ compatible, we do not actually have to test compatibility  with respect to {\it all } homomorphisms $\varphi: R^{1/p^e} \rightarrow R$. Since $ \Hom_R(R^{1/p^e},R)$ is a finitely generated $R^{1/p^e}$-module, it is enough to check compatibility with respect to a finite set of $R^{1/p^e}$-generators for each $e$. In the Gorenstein case, this takes an especially nice form: 
\begin{lemma}\label{generator}
If $(R, m)$ is an F-finite Gorenstein local ring, then $ \Hom_R(R^{1/p},R)$ is a cyclic $R^{1/p}$-module. 
An ideal $J$ is uniformally compatible if and only if it is compatible with respect to an $R^{1/p}$-module generator for  
$ \Hom_R(R^{1/p},R).$ \end{lemma}

\begin{proof} 
The point is that  if $R \rightarrow S$ is a finite map of rings with canonical module, there is an $S$-module isomorphism 
$\omega_S \cong  \Hom_R(S, \omega_R)$ \cite[Thm 3.7.7]{BrunsHerzogCMrings}.   If $R$ is a Gorenstein local ring, then $R$ is a canonical module for $R$ (and so of course ${R^{1/p}}$ is a canonical module for $R^{1/p}$),  so the first statement follows. 

By the same argument,  each  $ \Hom_R(R^{1/p^e},R)$ is a cyclic $R^{1/p^e}$-module. 
Moreover, if $\Psi$ is a generator for $ \Hom_R(R^{1/p},R),$ one easily checks that 
the composition map 
$$
\Psi_e = \Psi \circ \Psi^{1/p} \circ \dots \circ \Psi^{1/p^{e-1}}
$$ is a $R^{1/p^e}$ generator for $ \Hom_R(R^{1/p^e},R).$ [For example, it is easy to check that $\Psi_e$ is not in $m_{R^{1/p^e}}\Hom_R(R^{1/p^e},R),$ so it must be a generator by Nakayama's Lemma.] For example, $\Psi_2$ is the composition
$$
R^{1/p^2}  \,\,\overset{\Psi^{1/p}}\longrightarrow \,\,R^{1/p} \,\,\overset{\Psi}\longrightarrow \,\,R
$$
$$
r^{1/p^2} \mapsto [\Psi(r^{1/p})]^{1/p}  \mapsto \Psi ([\Psi(r^{1/p})]^{1/p}).
$$

Now, consider an ideal $J$ which is $\Psi$-compatible. Any $\varphi \in  \Hom_R(R^{1/p^e},R),$ 
can be written as $ \Psi_e \circ r^{1/p^e}$. So 
$$
\varphi(J^{1/p^e}) =  \Psi_e (r^{1/p^e}J^{1/p^e}) \subset  \Psi_e (J^{1/p^e}),$$
which by definition of $\Psi_e$ is the same as 
$$
 \Psi_{e-1} ( \Psi^{1/p^{e-1}} (J^{1/p^e})) =  \Psi_{e-1} ( [(\Psi (J^{1/p})]^{1/p^{e-1}}). 
 $$
 This   is contained in 
 $ \Psi_{e-1} (J^{1/p^{e-1}} ) $ because $\Psi (J^{1/p}) \subset J$ by $\Psi$-compatibility of $J$. 
 Finally, this is contained in $J$
by induction on $e$.  Thus any $\Psi$-compatible ideal is uniformly $F$ compatible. 
 \end{proof}
 
The generator in Lemma \ref{generator} is uniquely defined up to multiplication by a unit in $R^{1/p^e}$. It is sometimes abusively called the {\bf ``trace of the Frobenius map"}. 

For any  F-finite{\footnote{F-finite rings always admit a canonical module \cite[13.6]{GabberTStructures}.}} ring $R$, we can dualize the Frobenius map $R\rightarrow R^{1/p}$ into $\omega_R$: 
$$
\Hom_R(R^{1/p},\omega_R) \longrightarrow  \Hom_R(R,\omega_R)
 $$
 which produces an $R$-module map 
 $$
F_*\omega_{R^{1/p}} \rightarrow \omega_R,
 $$
 the {\it trace} of Frobenius.{\footnote{This map is only as canonical as the choice of $\omega_R$, so the ``the" is slightly misleading. Of course in geometric situations where the canonical module is defined by differential forms, there is a canonical choice.}}  In notation more common in algebraic geometry: the dual of the Frobenius map $R \rightarrow F_*R$ is the trace map $F_*\omega_R \rightarrow \omega_R$.  
 For smooth projective varieties, this is called the Cartier map.   If $R$ is local and Gorenstein, of course, this can be identified with a map 
 $F_*R \rightarrow R$, which will be a generator for  $ \Hom_R(R^{1/p^e},R).$ 
 
 Using $\omega_X$ has the advantage of globalizing; we have already encountered this idea in Section \ref{antican}. See    \cite{SchwedeTuckerTestidealsFiniteMorphism} or  \cite{BrionKumar}  for more on the trace map.

 \begin{ex}The trace of the Frobenius map is {\it not} usually a Frobenius splitting!
For example, if $R = \mathbb F_p[x, y]$, we recall that 
the monomials $x^ay^b$ where $0 \leq a, b, \leq p-1$ form a basis for $F_*R$ over $R$.
As the trace map, we can take
the $R$-linear map 
$
R^{1/p} \overset{\Psi}\longrightarrow R$ 
sending 
$ (xy)^{\frac{p-1}{p}} $ to $1$ and all other monomials in the basis to zero. This is clearly {\it not} a Frobenius splitting. 
The Frobenius splitting of Example \ref{ex3} can be obtained as $\Psi \circ (xy)^{\frac{p-1}{p}}$.
\end{ex}

\begin{Rmk}\label{Cartier}
 Blickle's  {\it Cartier algebras}   give another point of view on test ideals and uniformly $F$ compatible ideals   \cite{BlickleTestIdealsp-eLinearMaps}. 
 An $R$-module map $R^{1/p^e}\rightarrow R$ can be viewed  as an additive map $R\overset{\phi}\longrightarrow R$ satisfying 
 $\phi(r^{p^e} x) = r\phi(x)$ for any $r, x \in R$.   Blickle and B\"ockle dub this a $p^{-e}$-linear map
 \cite{BlickleBockleCartierModulesFiniteness}. 
 This point of view has the advantage that composition is slicker---  the source and target are always $R$---so we can easily compose such maps.  Indeed, the composition of $p^{-e}$ and $p^{-f}$ linear maps is easily checked to be $p^{-e-f}$-linear. 
  The  {\bf Cartier algebra}{\footnote{Here we assume that $R$ is reduced and of dimension greater than zero. In general, the definition of Cartier algebra is slightly more technical, but it reduces to this under very mild conditions. See  \cite{BlickleTestIdealsp-eLinearMaps}. }}
    $\mathcal C(R)$ is the subalgebra of 
 $\Hom_{\mathbb Z}(R, R)$ generated by all $p^{-e}$-linear maps (as we range over all $e$). 
Clearly $R$ is a module over  $\mathcal C(R)$, and  clearly its   $\mathcal C(R)$-submodules are precisely the uniformly F-compatible ideals.
The trace map can also be easily interpreted in this language: 
 in the Gorenstein local case, the trace $\Psi_e$ of Lemma \ref{generator} is literally the composition of $\Psi$ with itself $e$-times, so that  $\Psi$ generates $\mathcal C(R)$ as an $R$-algebra.   Blickle develops the {\it ultimate generalization of test ideals} in \cite{BlickleTestIdealsp-eLinearMaps} by looking at submodules of $R$ and other modules under various distinguished subalgebras of (variants of)  the Cartier algebra. \end{Rmk}

\begin{Rmk}\label{morehistory}
The  uniformly $F$ compatible ideals have been studied for many years in the tight closure literature under many different names. They were first studied by Smith in the local Gorenstein case \cite{SmithRationalSingularities} where they are called {\it F-ideals}, and soon after for more general local rings in \cite{LyubeznikSmith2},
where they are descriptively called  annihilators of $\mathcal F$-submodules of $E$. Both these papers have a  dual point of view  to our current perspective, which was first proposed by Schwede in   \cite{SchwedeCentersOfFPurity}, and it is not obvious that the definitions there produce precisely the uniformly $F$ compatible ideals 
(see \cite[Thm 4.1]{EnescuHochsterTheFrobeniusStructureOfLocalCohomology} for a proof).  Schwede used the term  
{\it  uniformly F-compatible} ideals, in a nod to the connection with Mehta and Ramanathan's notion of compatibly Frobenius split subschemes. In  \cite{SchwedeCentersOfFPurity}, Schwede shows how the prime compatible ideals (which he calls centers of sharp F-purity) can be viewed characteristic $p$ analogs of log canonical centers.  
The term $\phi$-compatible is lifted from the survey \cite{SchwedeTuckerSurvey}.
Generalizations of uniformly $F$ compatible ideals also come up in the work of Blickle ({\it e.g.} \cite{BlickleTestIdealsp-eLinearMaps})  under the name of Cartier-submodules and crystals; see the survey \cite{BlickleSchwedeSurvey}.
\end{Rmk}

  \section{Test Ideals for Pairs.}
As deeper connections between Frobenius splitting and singularities in birational geometry emerged, it was natural to look for generalizations of the characteristic $p$ story to  ``pairs," the natural setting for much of the geometry. For example, with the realization that the multiplier ideal ``reduces mod $p$ to the test ideal" (when the former is defined; see \cite{SmithMulplierIdealUniversalTestIdeal}, \cite{HaraGeometricInterpTestideals}), interest rose in defining test ideals for pairs, since this was the main setting for multiplier ideals. 
After Hara and Watanabe introduced Frobenius splitting for  pairs \cite{HaraWatanabe}, the 
   theory  of tight closure for pairs quickly developed in a series of technical  papers  by the Japanese school of tight closure, beginning about the time of the last decade's special year in commutative algebra at MSRI. In particular, a theory of test ideals for pairs was introduced by Hara, Yoshida and Takagi   \cite{HaraYoshidaGeneralizationOfTightClosure} and \cite{HaraTakagiGeneralizationTestIdeals}.

In this lecture, we introduce the theory of test ideals for pairs, focusing on the case where the ambient variety is smooth and affine---the ``classical algebro-geometric setting."
We do not  use the traditional tight closure definition,  but rather an equivalent definition first proposed in \cite{BMS1}. By shunning the most general setting, and  instead working in the simplest useful setting, we hope to highlight  the elegance of  test ideal arguments when the ambient ring is regular. In particular, we  give elementary proofs of all the basic properties, several of which do not seem to have been noticed before.  As an application, we include a self-contained proof of a well-known theorem on the behavior of symbolic powers of ideals in a regular ring following the analogous multiplier ideal proof in \cite{ELS1}.   Cf. \cite{HaraCharPMultiplierIdeals}.

 Let $R$ be an F-finite domain, and  let $\fa$ be an ideal of $R$. For each non-negative real number $t$, we associate an ideal{\footnote{If $t$ is a natural number, the notation $\tau(R,\fa^t)$ could be interpreted to  mean the test ideal of the ideal $\fa^t$ (with exponent 1) or to mean the test ideal of the ideal $\fa$ with exponent $t$. Fortunately, these  ideals are the same (as will soon be revealed when we give the definition), so  the danger of confusion is minimal.}}
 \[\{t\in\R_{\geq 0}\}\ \rightsquigarrow\ \{\tau(R,\fa^t)\}_{\R_{\geq 0}}.\]
In the classical commutative algebra case, $\fa=R$, and all $\tau(R,\fa^t) $ produce the same ideal, $\tau(R),$  the test ideal discussed in the previous lecture. 
For many years, this was the only test ideal in the literature. 
In the classical algebraic geometry setting, multiplier ideals are much easier to handle in the case where the ambient variety is {\it smooth}; indeed  the emphasis has always been that case. So perhaps it should not be surprising that, returning the commutative algebra to the case  where the ambient ring $R$ is regular, arguments should simplify dramatically for test ideals as well. 

\subsection{Test ideals in ambient regular rings.}
Let $R$ be an F-finite regular domain, and $\fa$ any ideal of $R$. 
We first define the test ideal of a pair  $\tau(R,\fa^t)$ in the special case where  $t$ is a positive rational number whose denominator is a power of $p$. The case of arbitrary $t$ will be obtained by approximating $t$ by a sequence of rational numbers whose denominators are powers of $p$.  When $R$ is clear from the context, we will often write $\tau(\mathfrak a^t)$.

For each $R$-linear map  $\phi: R^{1/p^e} \rightarrow R$, we  consider the image of $\fa$ under $\phi$. That is, looking at the ideal ${\fa}^{1/p^e}$ as an $R$-submodule of
 $R^{1/p^e}$, we consider its image
 $\phi(\fa^{1/p^e}) \subset R$,  which is of course an ideal of $R$. Ranging over all $\phi \in \Hom_R(R^{1/p^e}, R)$, we get the test ideal of $\fa$. Before stating this formally,  we set up  some notation:
 $$
 \fa^{[1/p^e]} := \sum_{\phi \in \Hom_R(R^{1/p^e}, R)} \phi(\fa^{1/p^e}).
 $$
\begin{lemma}\label{increase}
For any ideal $\fa$ in a Frobenius split ring $R$, we have
$$
\fa^{[1/p^e]} \subset (\fa^p)^{[1/p^{e+1}]}, 
$$ with equality if $\fa$ is principal.
\end{lemma}
\begin{proof}
Fix a Frobenius splitting $\pi:R^{1/p} \rightarrow R$. There is a corresponding splitting 
$$
\pi^{1/p^{e}}:  R^{1/p^{e+1}} \rightarrow R^{1/p^{e}}
$$ 
defined by 
taking $p^e$-th roots of everything in sight. For any $R$-linear map $R^{1/p^e} \overset{\phi}\rightarrow R$,  the composition 
$$
R^{1/p^{e+1}}\overset{\pi^{1/p^e}}\longrightarrow R^{1/p^e} \overset{\phi}\rightarrow R
$$
is an element of $\Hom_R(R^{1/p^{e+1}}, R)$. Now, the 
 ideal  $\fa^{[1/p^e]}$  is generated by elements of the form $\phi(x^{1/p^e}) $, where $x \in \fa$.
 To see that all such elements are also  in 
 $(\fa^p)^{[1/p^{e+1}]}$, note simply that $x^p$ is in $\fa^p,$ and that 
 the composition above sends 
$ (x^p)^{1/p^{e+1}}$ to $ \phi(x)$. This completes the proof. 
\end{proof}

Now, given a rational number $t $ whose denominator is a power of $p$, we can write 
$ t = \frac{n}{p^e} = \frac{np}{p^{e+1}} = \frac{np^2}{p^{e+2}} = \cdots $. The lemma implies  a corresponding increasing sequence of ideals:
\begin{equation} \label{seq}
(\fa^n)^{[1/p^e]} \subset (\fa^{np})^{[1/p^{e+1}]}  \subset (\fa^{np^2})^{[1/p^{e+2}]} \subset \cdots
\end{equation} which must eventually stabilize by the Noetherian property of the ring.
This sequence stabilizes to the  test ideal: 

\begin{defi}
\label{definition of test ideal: special case}
Let $R$ be an F-finite  regular ring of characteristic $p$   and let $\fa$ be an ideal of $R$. 
Fix any positive rational number whose denominator is a power of $p$, say $\frac{n}{p^e}.$
  The test ideal 
\[\tau(\fa^{\frac{n}{p^e}}) =  \bigcup_f (\fa^{np^f})^{[1/{p^{e+f}}]} =  
\sum_{\varphi\in \Hom_R(R^{1/p^{e+f}},  R)}\varphi((\fa^{np^{f}})^{1/p^{e+f}}).\]
That is, if we write the number $t = n/p^e$ with a sufficiently high power of $p$ in the denominator,
the test ideal $\tau(\fa^{\frac{n}{p^e}})$ is the ideal 
$(\fa^n)^{[1/p^e]}$ 
of $R$ generated by the images of the ideal $(\fa^n)^{1/p^e} \subset R^{1/p^e}$ under all projections  $R^{1/p^e} \rightarrow R.$ 
\end{defi}

\begin{Rmk} If $R$ is Frobenius split and $\fa$ is principal,  the sequence \ref{seq} above stabilizes immediately. That is, $ \tau(\fa^{\frac{n}{p^e}}):= (\fa^{n})^{[1/{p^{e}}]}$ for any representation of the fraction $\frac{n}{p^e}$.  \end{Rmk}
\begin{Rmk}
\label{exercise: test ideal of the maximal ideal in two variables}
Let $R=\F_p[[x,y]]$ and $\fa=(x,y)$. Show that  $\tau(\fa^n) = (x, y)^{n-1}$ for all $n \in \mathbb N$. \end{Rmk}

\subsubsection{The case of arbitrary $t$.}  Fix a positive real number $t$.   Choose a  {\it non-increasing\/} sequence of rational numbers $\{ t_n\}$ whose denominators are $p$-th 
powers.{\footnote{For example, we can take Hernandez's  sequence of successive truncations of a non-terminating base $p$ expansion for $t$; see \cite{HernandezFPTDiagonal}.} This allows us to define the test ideal $\tau(\fa^t)$ because for $t_n$ sufficiently close to $t$, the ideals $\tau(\fa^{t_n})$ will all coincide.   Indeed, if $\frac{n}{p^e} > \frac{m}{p^f}$, then getting a common denominator,  we have that 
$  (\fa^{np^f})^{[1/p^{e+f}]} \subset (\fa^{mp^e})^{[1/p^{f+e}]}.$ So
 if $\frac{n}{p^e} > \frac{m}{p^f}$, then clearly
\begin{equation}\label{chain}
  \tau(\fa^{n/p^e}) \subset \tau(\fa^{m/p^f}). 
  \end{equation}
Thus any decreasing sequence of positive rational numbers whose denominators are $p$-th powers must produce an ascending chain of ideals, which stabilizes by the Noetherian property of the ring. If two such descending sequences converge to the same real number $t$, it is clear again  by property (\ref{chain}) that the corresponding chains of ideals must stabilize to the {\it same} ideal. 
Thus we can define:
\begin{defi}
\label{definition of test ideal}\cite{BMS1}
Let $R$ be an F-finite  regular ring of characteristic $p$  and let $\fa$ be an ideal of $R$. For each $t\in \R_{\geq 0}$, we define
\[\tau(\fa^t):=\bigcup_{e\in\N}\tau(\fa^{\frac{\lceil tp^e\rceil }{p^e}}).\]
\end{defi}

The sequence $\frac{\lceil tp^e\rceil }{p^e}$, as $e$ runs through the natural numbers, is a decreasing sequence converging to the real number $t$. We  have picked it  for the sake of definitiveness; {\it any\/} such deceasing  sequence can be used instead, they all produce the same ideal in light of the inclusion \ref{chain}.

\subsection{Properties of Test Ideals.} All the basic properties of test ideals for an ambient regular ring follow easily from the definition, using the flatness of Frobenius for regular rings.
\begin{Thm}
\label{properties of test ideals}
Let $R$ be an F-finite  regular ring of characteristic $p$, with ideals $\fa,\fb$. The following properties of  the test ideal hold:
\begin{enumerate}
\item $\fa\subseteq \fb\ \Rightarrow\ \tau(R,\fa^t)\subseteq \tau(R,\fb^t)$ for all $t\in \mathbb R_{>0}$.
\item $t\geq t'\ \Rightarrow\ \tau(R,\fa^t)\subseteq \tau(R,\fa^{t'})$.
\item $\tau((\fa^n)^t)=\tau(\fa^{nt})$ for each positive integer $n$ and each $t\in\R_{> 0}$.
\item Let $W$ be a multiplicatively closed set in $R$, then
\[\tau(R,\fa^t)RW^{-1}=\tau(RW^{-1},(\fa RW^{-1})^t).\] 
\item Let $\overline{\fa}$ denote the integral closure of $\fa$ in $R$. Then
\[\tau(\overline{\fa}^t)=\tau(\fa^t)\ {\rm for\ all\ }t.\]
\item For each $t\in \R_{> 0}$, there exists an $\varepsilon>0$ such that $\tau(\fa^{t'})=\tau(\fa^t)$ for all $t'\in [t,t+\varepsilon)$.
\item $\fa\subseteq \tau(\fa)$.
\item {\rm (Brian\c con-Skoda\ Theorem{\footnote{In geometry circles, it is typical to refer to this statement as Skoda's Theorem; we adopt the more generous tradition of commutative algebra. This type of statement has also been referred to as a ``Brian\c con-Skoda theorem with coefficients." See {\it e.g.} \cite{AberbachHunekeBrianconSkodawithCoeff} or \cite{AberbachHosryLessRestrictiveBrainconSkoda}. } })} If $\fa$ can be generated by $r$ elements, then for each integer $\ell\geq r$ we have
\[\tau(\fa^{\ell})=\fa\tau(\fa^{\ell-1}).\]
\item {\rm (Restriction\ Theorem)} Let $x\in R$ be a regular parameter and $\fa\ {\rm mod\ }x$ denote the image of $\fa$ in $R/(x)$, then
\[\tau((\fa\ {\rm mod\ }x)^t)\subseteq \tau(\fa^t)\ {\rm mod\ }x.\]
\item {\rm (Subadditivity\ Theorem)} 
If $R$ is essentially of finite type over a perfect field, then 
$\tau(\fa^{tn})\subseteq \tau(\fa^t)^n$ for all $t\in \R_{\geq 0}$ and all $n \in \mathbb N$.{\footnote{More generally, our proof of the subadditivity property shows that for  the {\it mixed test ideal}  $\tau(\fa^t\fb^s) $ defined  analogously as $\tau(\fa^{\lceil{sp^e\rceil}}\mathfrak b^{\lceil{tp^e}\rceil})^{[1/p^e]}$ for $e\gg 0$, we have
 $\tau(\fa^t \fb^s)\subseteq \tau(\fa^t)\tau(\fb^s)$ for all $t,s\in \R_{\geq 0}$.}}
\end{enumerate}
\end{Thm}

\begin{Rmk}
In fact, the first six properties above hold more generally;  this is basic for the first five, once the definitions have been made (see \cite{SchwedeTuckerSurvey}) and the sixth is proved in 
\cite{BSTZ}.
 There are various versions of the other properties as well in more general settings, but most require some sort of restriction on the singularities of $R$ and the proofs tend to be very technical; See {\it e.g.} \cite{HaraYoshidaGeneralizationOfTightClosure}, \cite{TakagiYoshidaSymbolicPowers}, \cite{TakagiFormulasForMultiplierIdeals}. 
Our proofs mostly follow \cite{BMS1}. The simple proof of (9) here is new (hence so is the proof of the corollary (10)),  although the statements  can be viewed as (very) special cases of much more technical results  in 
\cite{TakagiCharPpltSing} and  \cite[Th 6.10(2)]{HaraYoshidaGeneralizationOfTightClosure}, respectively. \end{Rmk}
\begin{proof}
\label{exercise: prove properties of test ideals}
The first three properties follow immediately from the definitions, and the fourth is also straightforward. These are left to the reader.

The fifth  follows easily from the following basic property of integral closure (see eg. \cite[Cor 1.2.5] {HunekeSwansonIntegralCLosurebook}): there exists a natural number  $\ell$ such that for all $n \in \mathbb N$,
 $(\overline\fa)^{n+\ell} \subset \fa^n.$  Fix this $\ell$.  We already know that $\tau(\fa^t) \subset \tau(\overline{\fa}^t)$ by property (1). For the reverse inclusion, note 
 that since   $\frac{\lceil tp^e\rceil + \ell}{p^e}$ is a decreasing sequence converging to $t$ as $e $ gets large, we have that 
 $$
\tau({\overline{\fa}}^t ) =  \bigcup_{e\in\N}\tau({\overline \fa}^{\frac{\lceil tp^e\rceil + \ell}{p^e}})
=  \bigcup_{e\in\N}\tau(({\overline \fa}^{\lceil tp^e\rceil + \ell})^{\frac{1} {p^e}})
\subset 
  \bigcup_{e\in\N}\tau(({\fa}^{\lceil tp^e\rceil })^{\frac{1} {p^e}})
=
 \bigcup_{e\in\N}\tau(\fa^{\frac{\lceil tp^e\rceil }{p^e}}) =  \tau(\fa^t),
 $$  
with the inclusion coming from property (1). 

The sixth follows immediately from the Noetherian property of the ring.
Since $\frac{\lceil p^et\rceil +1}{p^e}$ is a decreasing sequence converging to $t$, we can fix 
 $e$ large enough that  $\tau(\fa^t)$ agrees with $\tau(\fa^{\frac{\lceil p^et\rceil +1}{p^e}})$. 
By Property (2), for all  $t'$ in the interval $[t, \frac{\lceil p^et\rceil+1}{p^e}), $
 we have $\tau(\fa^{\frac{\lceil p^et\rceil+1}{p^e}}) \subset  \tau(\fa^{t'}) \subset \tau(\fa^t)$.  In other words, all three ideals are the same. 

The seventh property is easy too: since $t = p^e/p^e$ for all $e$, we have 
$\tau(\fa) = (\fa^{p^e})^{[1/p^e]}$ for $e \gg 0$, which contains $\fa$ by Lemma \ref{increase}.

The Briancon-Skoda property is also easy. Thinking of  $\ell$ as $\frac{\ell p^e}{p^e}$ for large $e$, we have
$\tau(\fa^{\ell}) = (\fa^{\ell p^e})^{[\frac{1}{p^e}]}$. But it is easy to see that  $\fa^{\ell p^e} = 
\fa^{[p^e]}   (\fa^{(\ell-1) p^e}), $ where $\fa^{[p^e]} $ is the ideal generated by the $p^e$-th powers of the elements of $\fa$.
 [Indeed, 
if $\fa$ is generated by the  elements $a_1, \dots, a_r$, then $\fa^{\ell p^e}$ is generated by the monomials $a_1^{i_1} \dots a_r^{i_r}$ of degree $\ell p^e$;   if all exponents $i_j \leq p^e-1$, then the $p^e\ell \leq rp^e-r$, contradicting our assumption that $\ell \geq r$.]
Now clearly
$$\tau(\fa^{\ell}) = (\fa^{\ell p^e})^{[\frac{1}{p^e}]} =  ( \fa^{[p^e]}   \fa^{(\ell -1) p^e})^{[\frac{1}{p^e}]} 
=  \fa ( \fa^{(\ell -1) p^e})^{[\frac{1}{p^e}]}  =
\fa \tau (\fa^{\ell-1}) .
$$
The third equality here holds since by definition, for any ideal $\mathfrak b$, we have $\mathfrak b^{[1/p^e]}$ is the image of $\mathfrak b^{1/p^e}$ under the $R$-linear maps $R^{1/p^e} \rightarrow R$. In particular, 
$(\fa^{[p^e]}\mathfrak b)^{[1/p^e]} = \fa  \mathfrak b^{[1/p^e]}.$

Now the restriction property (9). Let us denote $R/(x) $ by  $\overline{R};$ its elements are denoted $\overline r$ where $r $ is any representative in $ R$. Consider any $\overline{R}$-linear map $\overline{R}^{1/p^e} \longrightarrow \overline{R}.$
We claim that this map lifts to a $R$-linear map $\phi: R^{1/p^e} \rightarrow R.$   Indeed, 
consider  the  diagram of $R$-modules
\[\xymatrix{
R^{1/p^e} \ar@{.>}[d] \ar@{>>}[r]& \overline{R}^{1/p^e} \ar[d]\\
R \ar@{>>}[r] & \overline{R}
}\]
where the horizontal arrows are the natural surjections, and the vertical arrow is the one we are given.  Because the bottom arrow is surjective and  $R^{1/p^e}$ is a {\it projective} $R$ module (by Kunz's Theorem \ref{kunzthm}), the composition map $R^{1/p^e} \rightarrow \overline{R}$ lifts to some $\phi:R^{1/p^e} \rightarrow R$ making the diagram commute.  Thus it is reasonable to denote the given map  $\overline{R}^{1/p^e} \longrightarrow \overline{R}$ by $\overline\phi
$. For any $\overline r \in \overline R$, we have $\overline{\phi} (\overline r^{1/p^e}) = 
\overline{\phi(r^{1/p^e})}$.

With this observation in place, the restriction theorem is easy. Take any $\overline y \in \tau (\overline{\fa}^t)$. By definition, there is some $\overline \phi: \overline R^{1/p^e} \rightarrow \overline R$ such that $\overline y = \overline{\phi}({\overline r}^{1/p^e})$, where $\overline r \in \overline \fa^{t}$. By the commutativity of the diagram,   $\overline y = \overline{{\phi}(r^{1/p^e})},$ for some $r \in \fa^{t}$. 
That is, $\overline y \in \tau(\fa^t) \,\, mod\, (x)$. The restriction theorem is proved. 

Finally, we observe that the subadditivity property follows formally from the restriction property in exactly the same way as for multiplier ideals; Cf \cite{BlickleLazarsfeldMultiplier}. Let $\fa$ and $\mathfrak b$  be ideals in a regular ring $R$ essentially finitely generated over $k$. In $S = R \otimes_k R$, which is also regular, we have the ideal $\fa \otimes R + R \otimes \mathfrak b$. For any positive rational $s, t $, it is easy to check that 
$$
\tau(\fa^s \otimes R + R \otimes \mathfrak b^t) =  \tau (\fa^s) \otimes R + R \otimes \tau( \mathfrak b^t).
$$ 
Now, locally at each maximal ideal, the diagonal ideal $\Delta \subset R\otimes R$ is generated by a sequence of regular parameters. By the restriction property, at each maximal ideal we have
$$
\tau((\fa^s \otimes R + R \otimes \mathfrak b^t) \, mod \, \Delta) \subset [ \tau (\fa^s) \otimes R + R \otimes \tau( \mathfrak b^t) ] \, mod\, \Delta.
$$ 
Interpretting this in $R\otimes R / \Delta = R$ yields the inclusion
$
\tau(\fa^s \mathfrak b^t) \subset  \tau (\fa^s)\tau( \mathfrak b^t)
$.
\end{proof}
\subsection{Asymptotic Test Ideals and an Application to Symbolic Powers.}
We now introduce an asymptotic version of the test ideal, analogous to the asymptotic multiplier ideal first defined in \cite{ELS1}. We will use this concept to give a simple  proof of  the following well-known theorem about the asymptotic behavior of symbolic powers.

\begin{Thm}[Ein-Lazasfeld-Smith; Hochster-Huneke]
\label{uniform symbolic power in polynomial ring}
Let $I$ be an unmixed ({\it e.g.} prime) ideal in $k[x_1,\dots,x_d]$. Then 
\[I^{(dn)}\subseteq I^n\ {\rm for\ all\ }n\in\N.\]
\end{Thm}

 Our proof here is a straightforward and self-contained adaptation of the original multiplier ideal proof in \cite{ELS1} in characteristic zero. Hara  had also adapted that proof to prime characteristic using   using  test ideals  in  \cite{HaraCharPMultiplierIdeals}, although the definitions and proofs are different and less self-contained than ours here.  Hochster and Huneke gave a tight closure proof and generalized this result in the characteristic $p$ case \cite{HochsterHunekeSymbolicOrdinaryPowersIdeals}.   See also  \cite{TakagiYoshidaSymbolicPowers}.  

For a prime  ideal $\mathfrak p$  in a polynomial ring, the symbolic power 
$\mathfrak p^{(n)}$ is  the ideal of all functions vanishing to order $n$ on the variety defined by $\mathfrak a$. Put differently, the symbolic powers of a prime ideal $\mathfrak p$  in any ring $R$ are defined by  $\mathfrak p^{(n)} = \mathfrak p^nR_{\mathfrak p} \cap R. $ 
For  arbitrary $\mathfrak a$,
we take a primary decomposition $\mathfrak a^N=\mathfrak p_1\cap \cdots \cap \mathfrak p_n\cap Q_1\cap \cdots \cap Q_m$ where $P_i$'s are the minimal primary components and $Q_j$'s are the embedded components, then define  $\mathfrak a^{(N)}=\mathfrak p_1 \cap \cdots \cap \mathfrak p_n$.

 \begin{defi}
 A sequence of ideals $\{\fa_n\}_{n\in \N}$ is called a {\it graded sequence} of ideals if \[\fa_n\fa_m\subseteq \fa_{n+m}\] 
 for all $n,m$. 
 \end{defi} 
It is easy to check that the symbolic powers
 $\{\fa^{(n)}\}_{n\in \N}$
 of any ideal  $\fa$ in any ring form a graded sequence. 
 Graded sequences arise naturally in many contexts in algebraic geometry. For example, the sequence of base loci of the powers of a fixed line bundle form a graded sequence of ideals on a variety. See \cite{ELS1} \cite{ELS2} or \cite{BlickleLazarsfeldMultiplier} for many more examples. 
 
Given any graded sequence of ideals $\{\fa_n\}$, it follows from the definition and Property \ref{properties of test ideals}(1) that for any positive $\lambda$,  \[\tau(\fa_n^{\lambda})=\tau((\fa_n^{\lambda m})^{1/m})\subseteq \tau(\fa_{mn}^{\lambda/m}).\] In other words, 
 the collection $$\{\tau(\mathfrak a_m^{\lambda/m})\}_{m\in \mathbb N}$$ has the property that any two ideals are dominated by a third in the collection. 
Since $R$ is noetherian, this collection must have a maximal element; this stable ideal is called the {\bf asymptotic test ideal:}
\begin{defi}
The $n$-th asymptotic test ideal of the graded sequence $\{\fa_n\}_{n\in \N}$ is the ideal 
\[\tau_{\infty}(R,\fa_n) := \sum_{\ell \in \mathbb N} \tau(R,\fa^{1/\ell}_{\ell n}),\]
which is equal to 
$$\tau(R,\fa^{1/m}_{mn})$$
for sufficiently large and divisible $m$.
\end{defi}

By definition,  it is clear that $\tau_{\infty}(R,\fa_n)$ satisfies appropriate analogs of all the properties listed in Properties \ref{properties of test ideals}--- the asymptotic test ideal is a particular test ideal, after all. Especially we point out a consequence of the subadditivity theorem:
\begin{Cor}\label{subadd}
For any graded sequence in an F-finite regular ring $R$, we have $\tau_{\infty}(R,\fa_{nm}) \subset (\tau_{\infty}(R,\fa_n))^m$ for all $n, m \in \mathbb N$.
\end{Cor}

\begin{proof}
Since $\tau_{\infty}(R,\fa_{nm}) := \tau (R,\fa_{nm\ell}^{\frac{1}{\ell}})  $ for sufficiently divisible $\ell$, we have 
 $$\tau_{\infty}(R,\fa_{nm}) = \tau (R,\fa_{nm\ell}^{\frac{1}{\ell}})  
 =  \tau (R,\fa_{nm\ell}^{\frac{m}{m\ell}})
   \subset   \tau (R,\fa_{nm\ell}^{\frac{1}{m\ell}})^m,$$
 with the inclusion following from the subadditivity property \ref{properties of test ideals}(10) for test ideals. 
 Since $\ell$ here can be taken arbitrarily large and divisible, we have that 
 $\tau (R,\fa_{nm\ell}^{\frac{1}{m\ell}}) = \tau_{\infty}(R,\fa_{n}).$
 Thus 
  $$\tau_{\infty}(R,\fa_{nm}) \subset 
   \tau_{\infty} (R,\fa_{n})^{m}. $$
\end{proof}

\begin{proof}[Proof of Theorem \ref{uniform symbolic power in polynomial ring}]
We consider the graded sequence of ideals $\{I^{(n)}\}_{n\in \N}$. According to Properties \ref{properties of test ideals}(3), we have $I^{(dN)}\subseteq \tau_{\infty}(I^{(dN)})$. By Corollary \ref{subadd}, we have
\[\tau_{\infty}(I^{(dN)})\subseteq \tau_{\infty}(I^{(d)})^N\]
for all $N$.
Hence it is enough to check that $\tau_{\infty}(I^{(d)})\subseteq I$. For this, we can check at each associated prime $\fp$ of $I$, which means essentially that we can assume that $R$ is local and that $I$ is primary to the maximal ideal; that is,  we need to show that 
\[\tau_{\infty}(R_{\fp},(I^{d}R_{\fp})) \subset IR_{\mathfrak p}. 
\]
In $R_{\fp},$ there is a reduction of $I$ that can be generated by $\dim(R_{\fp})\leq d$ elements, and hence according to Properties \ref{properties of test ideals}(5) we may assume that $I$ itself can be generated by $d$ elements. Then  the Briancon-Skoda property \ref{properties of test ideals}(8) tells us
 \[\tau_{\infty}(R_{\fp},(I^{d}R_{\fp}))\subseteq I.\]
 This finishes the proof of our theorem.
\end{proof}

\subsection{The definition of the test ideal for a pair $(R, \fa^t)$ in general.}\label{generaltest}

The definition of the test ideal for a singular ambient ring can be adapted  to the general case of pairs. We include the definition for completeness without getting into details; see Schwede \cite{SchwedeCentersOfFPurity} or \cite{SchwedeTuckerSurvey} for more, including generalizations to ``triples."

\begin{defi}
\label{defn: test ideal in general}
Let $R$ be a  reduced F-finite ring  and let $\fa$ be an ideal of $R$. The test ideal $\tau(R,\fa^t)$ is defined to be the smallest ideal $J$ not contained in any minimal prime that satisfies
\[\varphi((\fa^{\lceil t(p^e-1) \rceil}J)^{1/p^e})\subseteq J\]
for all $\varphi\in \Hom_R(R^{1/p^e},R)$ (ranging over all $e\geq 1$).

In particular, the test ideal $\tau(R, \fa^t)$ is defined to be the smallest non-zero ideal $J$ not contained in any minimal prime that satisfies
\[\varphi(J^{1/p^e})\subseteq J\]
for all $\varphi \in \Hom_R(R^{1/p^e},R)$ which are in the sub-module consisting of homomorphisms obtained by first pre-composing with elements of in  $(\fa^{\lceil t(p^e-1) \rceil})^{1/p^e}$,  and all $e\geq 1$.
\end{defi}

Again, the existence of such smallest nonzero ideal is a  non-trivial statement; see \cite{SchwedeTuckerSurvey} and \cite{HaraTakagiGenerlizationTestIdeals} for the general proof. 

Although it is not completely obvious, if $R$ is regular, this gives the test ideal that we have already discussed in the previous subsection. Indeed, both definitions developed here are shown to agree with the tight closure definition of the test ideal in   \cite{HaraYoshidaGeneralizationOfTightClosure} in their respective introductory papers, \cite{BMS1} and \cite{SchwedeCentersOfFPurity}.

\subsubsection{Further Reading on Test Ideals.}
Much more is known about test ideals than we can discuss here, and the story of test ideals is very much still a work in progress. One notable paper is \cite{SchwedeTuckerTestidealsofNonPrincipalIdeals} which discusses a framework under which test ideals and multiplier ideals can be constructed in the same way, an idea begun in 
 \cite{BlickleSchwedeTuckerAlteration}.  On the other hand, the paper \cite{MustataYoshidaTestvsMultiplier} includes a cautionary result: {\it every ideal in a regular ring is the test ideal of some ideal with some coefficient.} This indicates that test ideals are in some ways very different from multiplier ideals, since multiplier ideals are always integrally closed; see also \cite{McDermottTestidealDiagonalII}.

A rich literature has evolved on the study of F-jumping numbers---analogs of the {\it jumping numbers for multiplier ideals} of \cite{ELSVJumpingCoeff}. As with multiplier ideals,  as we increase the exponent  $t$, the ideals $\tau(R, \mathfrak a^t)$ get deeper;  the values of $\alpha$  such that $\tau(R, \mathfrak a^{\alpha -\epsilon})$ strictly contains 
$\tau(R, \mathfrak a^{\alpha})$ (for all positive $\epsilon$) are called {\it F-jumping numbers.} The smallest F-jumping number is called the F-pure threshold. 
First introduced in 
\cite{HaraYoshidaGeneralizationOfTightClosure}, one of the main questions has been whether or not the F-jumping numbers are always discrete and rational. The first major progress was the case of  regular ambient rings  \cite{BMS1};  the paper  \cite{SchwedeTuckerTestidealsofNonPrincipalIdeals} gives an exceptionally well-written account of the state of the art.
See also  \cite{BMS2},  \cite{BSTZ},  \cite{KatzmanLyubeznikZhangOnDiscretenessAndRationality},  \cite{SchswedeTuckerZhangSingleAlteration}.
The  F-jumping numbers are notoriously difficult to compute; see \cite{HernandezFPTDiagonal}. Just as jumping numbers for multiplier ideals (in characteristic zero) are roots of the Bernstein Sato polynomial \cite{ELSVJumpingCoeff}, similar phenomena have been studied for F-jumping numbers; see, for example, \cite{MustataTakagiWatanabe}, \cite{BlickleStablerBSPolyandTestIdeals}, or  \cite{MustataBSPolyCharP}.

The connection between the test ideal and differential operators was first pointed out in  \cite{SmithDModStructureFSplitRings}, where it is shown that the test ideal is a D-module. There are deep connections between the lattice of uniformly F-compatible ideals and intersection homology $D$-module in characteristic $p$ \cite{BlickleIntersectionHomologyDModule}, and other works of Blickle and his collaborators.  See also \cite{SmithVandenBergh}.

\appendix
\section{So what does Cohen-Macaulay mean? } The property of Cohen-Macaulayness is so central to commutative algebra that the field  has  been jokingly called the ``study of Cohen-Macaulayness." Cohen-Macaulayness is also  important in algebraic geometry, representation theory and combinatorics, with many different characterizations. We briefly review three of these. See \cite{BrunsHerzogCMrings} for a more in depth discussion.

 First,  Cohen-Macaulay is a local property---meaning that {\it we can  define} a Noetherian ring $R$ to be  Cohen-Macaulay if all its local rings are Cohen-Macaulay. So we focus only on what it means for a {\it local\/} ring $(R, m)$ to be Cohen-Macaulay. 
 
 Alternatively, if the reader prefers graded rings, one can take $(R, m)$ to mean 
 an $\mathbb N$-graded ring $R$, finitely generated over its zero-th graded piece $R_0$ (a field).  In this case, $m$ denotes the unique homogenous maximal (or irrelevant) ideal of $R$.

\medskip
\noindent
{\bf  The standard textbook definition:}  A local ring $(R, m)$ is Cohen-Macaulay if it admits a regular sequence{\footnote{homogenous in the graded case}} of length equal to the dimension of $R$. 
 A sequence of elements $x_1, \dots, x_d$ is regular if $x_1$ is not a zero divisor of $R$, and the image of $x_i$ in $R/(x_1, \dots, x_{i-1})$ is not a zero divisor for $i = 2, 3, \dots, d$. ({\it cf.} \cite[Definition 1.1.1, Definition 2.1.1]{BrunsHerzogCMrings}. ) 
Another point of view on regular sequences is this: the Koszul complex on a set of elements $\{x_1, \dots, x_d\}$ is exact if and only if the elements form a regular sequence.
 
 Regular sequences are useful for creating induction arguments using long exact sequences induced from the short exact sequences 
$$
 0 \rightarrow R/(x_1, \dots, x_{i-1}) {\overset{\cdot x_i} \longrightarrow } R/(x_1, \dots, x_{i-1}) \rightarrow R/(x_1, \dots, x_{i})
\rightarrow 0.
$$ In algebraic geometry, say when $R$ is the homogeneous coordinate ring of a projective variety, this is the technique of ``cutting down by hypersurface sections." This works best when  the resulting intersections contain no embedded points---which is to say, the defining equations of the hypersurfaces form a regular sequence.

\noindent
{\bf  A possibly more intuitive definition:}  Let $R$ be an $\mathbb N$-graded algebra  which is finitely generated over $R_0 = k$. Recall that every such ring admits a {\it Noether Normalization:} that is, $R$ can be viewed as  finite integral extension of some (graded) {\it polynomial} subalgebra $A$.  Then $R$ is Cohen-Macaulay if and only if $R$ is free as an $A$ module.  For example, in the case of Example \ref{ex1}, the ring $R = S^G = \mathbb C[x^2, xy, y^2]$ can be viewed as an extension of the regular subring $A =  k [x^2, y^2]$. As an $A$-module, $R$ is free with basis $\{1, xy\}$. That is, every element of $R$ can be written {\it uniquely\/} as a sum
$a + b xy$, where $a$ and $b$ are polynomials in $x^2, y^2$.

If $(R, m)$ is not graded but is a {\it complete} algebra over a field, then an analog of Noether Normalization called the ``Cohen-Structure theorem" holds, which allows us to write $R$ as a finite extension of a power-series subring $A$. Again, $R$ is Cohen-Macaulay if and only if $R$ is {\it free\/} as an $A$-module. 

We remark that in both the graded and complete case, it is easy to find the regular subring $A$.
In the graded case, the $k$-algebra generated by any homogenous system of parameters will be  a Noether normalization.
Likewise, in the complete case, the power series sub algebra generated by any system of parameters will work.

Even if the local ring $(R, m)$ is not complete, this criterion of Cohen-Macaulayness can be adapted by {\it completing\/} $R$ at its maximal ideal: it is not hard to prove that 
a local ring $(R, m)$ is Cohen-Macaulay if and only if its completion $\hat R$ at the maximal ideal is Cohen-Macaulay. This follows  immediately from the definition of regular sequence: since $\hat R$ is a faithfully flat $R$-algebra, the sequence 
$$
 0 \rightarrow R/(x_1, \dots, x_{i-1}) {\overset{\cdot x_i} \longrightarrow } R/(x_1, \dots, x_{i-1}) \rightarrow R/(x_1, \dots, x_{i})
\rightarrow 0
$$
is exact if and only if the sequence  
$$
 0 \rightarrow \hat R/(x_1, \dots, x_{i-1}) {\overset{\cdot x_i} \longrightarrow } \hat R/(x_1, \dots, x_{i-1}) \rightarrow \hat R/(x_1, \dots, x_{i})
\rightarrow 0
$$
is exact.

\bigskip
\noindent
{\bf A cohomological definition well-loved by commutative algebraists:} The local or graded ring $(R,  m)$ is Cohen-Macaulay if and only if  the local cohomology modules $H^i_{m}(R)$ are all zero for $i < \dim R$.  We will not launch into a long discussion of local cohomology here, which is well-known to all commutative algebraists ({\it cf.} \cite[$\S$3.5]{BrunsHerzogCMrings}).  It suffices  to know that
local cohomology  has all the usual functorial properties of any cohomology  theory, so even if you don't know the precise definition, a passing familiarity with any kind of cohomology should suffice to follow the ideas in arguments  in many situations.


\bibliographystyle{skalpha}
\bibliography{CommonBib}
  
\end{document}